\documentclass[11pt]{amsart}

\usepackage{latexsym}
\usepackage{soul}

\usepackage{amsmath}
\usepackage{mathtools}
\usepackage{amsfonts}
\usepackage{amssymb}
\usepackage{lipsum}
\usepackage{enumitem}
\usepackage{morefloats}
\usepackage{float}
\usepackage[noend]{algorithmic} 
\usepackage{algorithm,caption}
\algsetup{indent=2em}

\usepackage{tikz}
\usepackage{color}
\usepackage{placeins}
\usetikzlibrary{positioning}
\usetikzlibrary{arrows,shapes}
\newtheorem{theorem}{Theorem}[section]
\newtheorem{lemma}[theorem]{Lemma}
\newtheorem{corollary}[theorem]{Corollary}
\newtheorem{proposition}[theorem]{Proposition}
\newtheorem*{lemma*}{Lemma}

\newtheorem{sublemma}{}[theorem]

\theoremstyle{definition}

\theoremstyle{remark}

\numberwithin{equation}{section}
\usepackage{stackengine}

\stackMath

\newcommand{\ba}{\backslash}

\begin{document}

\title[$2$-cographs]{Generalizing Cographs to $2$-cographs}

\author{James Oxley}
\address{Mathematics Department\\
Louisiana State University\\
Baton Rouge, Louisiana}
\email{oxley@math.lsu.edu}

\author{Jagdeep Singh}
\address{Mathematics Department\\
Louisiana State University\\
Baton Rouge, Louisiana}
\email{jsing29@math.lsu.edu}

\subjclass{05C40, 05C83}
\date{\today}

\begin{abstract}
 A graph in which every connected induced subgraph has a disconnected complement is called a cograph. Such graphs are precisely the graphs that do not have the 4-vertex path as an induced subgraph. We define a $2$-cograph to be a graph  in which the complement of every $2$-connected induced subgraph  is not $2$-connected. We show that, like cographs, 
 $2$-cographs can be recursively defined. But, unlike cographs, $2$-cographs are closed under induced minors. We
 characterize the class of non-$2$-cographs for which every proper induced minor is a $2$-cograph. We further find the finitely many members of this class whose complements are also  induced-minor-minimal non-$2$-cographs. 
\end{abstract}

\maketitle

\section{Introduction}

In this paper, we only consider simple graphs. Except where indicated otherwise, our notation and terminology will follow \cite{text}. An {\bf induced minor} of a graph $G$ is any graph $H$ that can be obtained from $G$ by a sequence of operations each consisting of a vertex deletion or an edge contraction. If $H \neq G$, then $H$ is  a {\bf proper induced minor} of $G$. A consequence of the fact that we consider only simple graphs here is that, when we write $G/e$ for an edge $e$ of a graph $G$, we mean the graph that we get from the multigraph obtained by contracting the edge $e$ by deleting all but one edge from each class of parallel edges.

A {\bf cograph} is a graph in which every connected induced subgraph has a disconnected complement. By convention, the graph $K_1$ is taken to be a cograph. Replacing connectedness by $2$-connectedness, we define a graph $G$ to be a {\bf $2$-cograph} if $G$ has no induced subgraph $H$ such that both $H$ and its complement, $\overline{H}$, are $2$-connected. Note that $K_1$ is a $2$-cograph.
Cographs have been extensively studied over the last fifty years (see, for example, \cite{jung, sein, corneil2}). They are also called $P_4$-free graphs due to following characterization \cite{corneil}.

\begin{theorem}
\label{cographs_characterisation}
A graph $G$ is a cograph if and only if $G$ does not contain the path $P_4$ on four vertices as an induced subgraph.
\end{theorem}

%In other words, $P_4$ is the unique non-cograph with the property that every proper induced subgraph of $P_4$ is a cograph. The main result of the paper identifies non-$2$-cographs that are minimal in a certain natural sense.

In Section 2, we show that $2$-cographs can be recursively defined, that every induced minor of a $2$-cograph is also a $2$-cograph, and that the complement of every $2$-cograph is also a $2$-cograph. % We also generalize cographs to $k$-cographs for $k$ exceeding two and note that, in that case, the class of $k$-cographs is not closed under contraction. 
In addition, we correct a result of Akiyama and Harary \cite{harary} that had claimed to characterize when the complement of a $2$-connected graph is $2$-connected.

Because the class of $2$-cographs is closed under induced minors, our initial goal was to find all non-$2$-cographs  with the property that every  proper induced minor is a $2$-cograph.  But, as we show in Section $3$, in contrast to Theorem \ref{cographs_characterisation}, there are infinitely many such non-$2$-cographs. However, we were able to determine all infinite families of such graphs. For all $k \geq 1$, let $M_k$ and $N_k$ be the graphs shown in Figures \ref{3_deletable_vertices} and \ref{4_deletable_vertices}, respectively. Let $M_k'$ and $N_k'$ be obtained from $M_k$ and $N_k$ by adding the edge $st$. Further, let $N_k''$ be the graph obtained from $N_k'$ by adding the edge $uz$; let $L_k$ be the graph shown in Figure $\ref{2_deletable_vertices}$; and, for all $j \geq 0$, let $F_j$ be the graph shown in Figure $\ref{James_new}$. The next two theorems are the main results of the paper.

\begin{theorem}
\label{main_new_addition}
Let $G$ be a graph that is not a $2$-cograph such that every proper induced minor of $G$ is a $2$-cograph. Then

\begin{enumerate}[label=(\roman*)]
    \item $|V(G)| \leq 16$; or
    \item $G$ is the complement of a cycle of length at least five; or
    \item for some positive integer $k$, the complement of $G$ is isomorphic to $F_{k-1}, L_k, M_k, M_k', N_k, N_k', or N_k''$.
\end{enumerate}
\end{theorem}

As we were unable to improve this bound of $16$ vertices and the task of finding induced-minor-minimal non-$2$-cographs with at most $16$ vertices seemed computationally infeasible, we were prompted to try to determine those graphs $G$ for which both $G$ and  $\overline{G}$ are induced-minor-minimal non-$2$-cographs. The following theorem proves that, up to isomorphism, there are finitely many such graphs $G$. Its proof occupies most of Section $4$.

\begin{theorem}
\label{main}
Let $G$ be a graph. Suppose that $G$ is not a $2$-cograph but that every proper induced minor of each of $G$ and $\overline{G}$ is a $2$-cograph. Then $5 \leq |V(G)| \leq 10$.
\end{theorem}

The unique $5$-vertex graph satisfying the hypotheses of the last theorem is $C_5$, the $5$-vertex cycle. In the appendix, we list all of the other graphs that satisfy these hypotheses. 
%We can restate the main result as follows. For graphs $G$ and $H$, define $H \le G$ if $H$ can be obtained from $G$ by a sequence of operations each consisting of a  vertex deletion, an edge contraction, or complementation. If $G$ is a $2$-cograph and $H \le G$, then $H$ is a $2$-cograph. Let $\mathcal G$ be the collection of non-2-cographs 

\begin{figure}
    \centering
    \begin{tikzpicture}[scale=0.7,colorstyle/.style={circle, draw=black!100,fill=black!20, thick, inner sep=2pt, minimum size=0.5mm}]

    \node (v2) at (-13,1)[colorstyle]{$v_2$};
    \node (yk) at (-11.5,2.5)[colorstyle]{$y_j$};
    \node (y1) at (-11,5)[colorstyle]{$y_1$};
    \node (y0) at (-11.5,7)[colorstyle]{$y_0$};
    \node (u2) at (-13,8.5)[colorstyle]{$u_2$};
    
    \node (x) at (-17,4.5)[colorstyle]{$x$};
    
    \node (u1) at (-17.5,6.5)[colorstyle]{$u_1$};
    
    \node (v1) at (-17.5,2.5)[colorstyle]{$v_1$};
    
    %\node (d1) at (-11, 4) [] {$\vdots$};

    \draw[thick] (y1)--(y0)--(u2)--(u1)--(x)--(v1)--(v2)--(yk);
    
    \draw[thick]
    (x)--(y0)
    (x)--(y1)
    (x)--(yk);

    \draw[thick,loosely dotted] (y1)--(yk);

    \node (n1) at (-8,1)[colorstyle]{$v_1$};
    \node (n2) at (-8,3.4)[colorstyle]{$w$};
    \node (n3) at (-8,5.8)[colorstyle]{$x$};
    \node (n4) at (-8,8.2)[colorstyle]{$u_1$};
    \node (m1) at (-6,1)[colorstyle]{$v_2$};
    \node (m2) at (-6,3.4)[colorstyle]{$z$};
    \node (m3) at (-6,5.8)[colorstyle]{$y$};
    \node (m4) at (-6,8.2)[colorstyle]{$u_2$};

    \draw[thick] (n1)--(n2)--(n3)--(n4)--(m4)--(m3)--(m2)--(m1)--(n1);
    \draw[thick] (n2)--(m2)
                 (n3)--(m3);
                 
     \draw[thick] (n2)--(m3)
                  (n3)--(m2);

    \node (x1) at (-3,1)[colorstyle]{$v_1$};
    \node (x2) at (-3,3.4)[colorstyle]{$w$};
    \node (x3) at (-3,5.8)[colorstyle]{$x$};
    \node (x4) at (-3,8.2)[colorstyle]{$u_1$};
    \node (y1) at (-1,1)[colorstyle]{$v_2$};
    \node (y2) at (-1,3.4)[colorstyle]{$z$};
    \node (y3) at (-1,5.8)[colorstyle]{$y$};
    \node (y4) at (-1,8.2)[colorstyle]{$u_2$};

    \draw[thick] (x1)--(x2)--(x3)--(x4)--(y4)--(y3)--(y2)--(y1)--(x1);
    \draw[thick] (x2)--(y2)
                 (x3)--(y3);
    
    \draw[thick] (y2)--(x3);

\node (n1) at (-14.5,-0.5){$F_j~(j \geq 0)$};

\node (n1) at (-7,-0.5){$H_1$};

\node (n1) at (-2,-0.5){$H_2$};

\end{tikzpicture}
\caption{The complements of the induced-minor-minimal non-$2$-cographs that are critically $2$-connected.}
\label{James_new}
\end{figure}
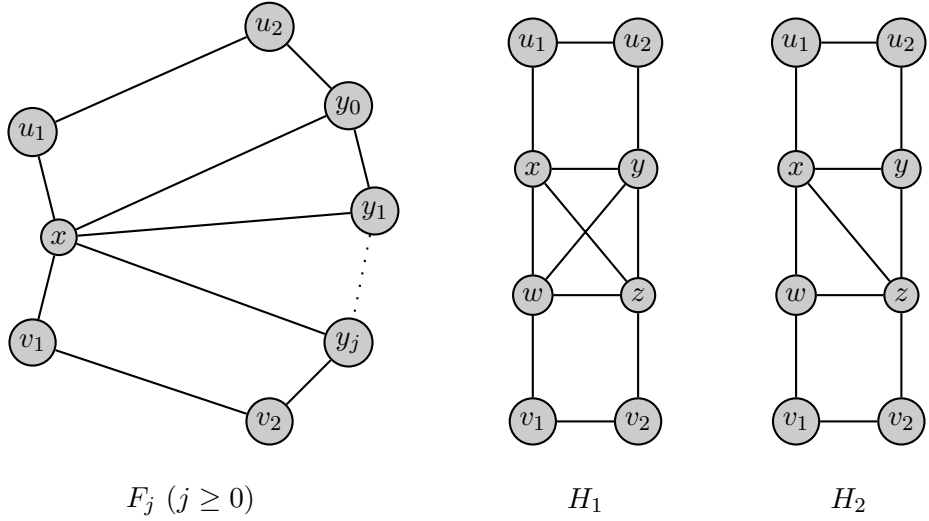

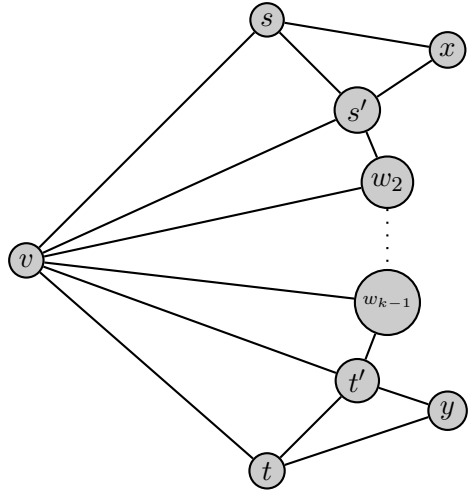
\begin{figure}
    \centering
    \begin{tikzpicture}[scale=0.8,colorstyle/.style={circle, draw=black!100,fill=black!20, thick, inner sep=2pt, minimum size=0.5mm}]

    \node (v2) at (-13,1)[colorstyle]{$t$};
    \node (yk) at (-11.5,2.5)[colorstyle]{$t'$};
    \node (y1) at (-11,3.8)[colorstyle]{\tiny $w_{k-1}$};
    \node (w1) at (-11,5.8)[colorstyle]{$w_2$};
    \node (y0) at (-11.5,7)[colorstyle]{$s'$};
    \node (u2) at (-13,8.5)[colorstyle]{$s$};
    
    \node (x) at (-17,4.5)[colorstyle]{$v$};
    
    \node (u1) at (-10,8)[colorstyle]{$x$};
    
    \node (v1) at (-10,2)[colorstyle]{$y$};

    \draw[thick] (y0)--(u2)--(u1)
    (v1)--(v2)--(yk);
    
     \draw[thick] (y0)--(w1)
     (yk)--(y1);
    
    \draw[thick]
    (x)--(y0)
    (x)--(y1)
    (x)--(yk)
    (x)--(v2)
    (x)--(u2)
    (u1)--(y0)
    (v1)--(yk)
    (x)--(w1);

    \draw[thick, loosely dotted] (w1)--(y1);

\end{tikzpicture}
\caption{For each $k \geq 1$, the complement of the above graph $L_k$ is an induced-minor-minimal non-$2$-cograph.}
\label{2_deletable_vertices}
\end{figure}

\begin{figure}
    \centering
    \begin{tikzpicture}[scale=0.7,colorstyle/.style={circle, draw=black!100,fill=black!20, thick, inner sep=2pt, minimum size=0.5mm}]

    \node (v2) at (-13,1)[colorstyle]{$u$};
    \node (yk) at (-11.5,2.5)[colorstyle]{$u'$};
    \node (y1) at (-11,4.5)[colorstyle]{\tiny $w_{k-1}$};
    \node (y0) at (-11.5,7)[colorstyle]{$w_1$};
    \node (u2) at (-13,8.5)[colorstyle]{$s$};
    
    \node (x) at (-17,4.5)[colorstyle]{$v$};
    
    \node (u1) at (-10,8)[colorstyle]{$t$};
    
    \node (v1) at (-10,2)[colorstyle]{$y$};
    
    \node (t) at (-11,9)[colorstyle]{$x$};
    
    \draw[thick] (y0)--(u2)
    (v1)--(v2)--(yk);
    
    \draw[thick]
    (t)--(u2)
    (y1)--(yk)
    (t)--(u1)
    (x)--(y0)
    (x)--(y1)
    (x)--(yk)
    (x)--(v2)
    (x)--(u2)
    (u1)--(y0)
    (v1)--(yk);
    
    \draw[thick] (x) to [out=100,in=90] (u1);

    \draw[thick, loosely dotted] (y0)--(y1);

\end{tikzpicture}
\caption{$M_k$, a graph whose complement is an induced-minor-minimal non-$2$-cograph.}
\label{3_deletable_vertices}
\end{figure}

\begin{figure}
    \centering
    \begin{tikzpicture}[scale=0.7,colorstyle/.style={circle, draw=black!100,fill=black!20, thick, inner sep=2pt, minimum size=0.5mm}]

    \node (v2) at (-13,1.4)[colorstyle]{$u$};
    \node (yk) at (-11.5,2.5)[colorstyle]{$w_k$};
    \node (y1) at (-11,5)[colorstyle]{$w_2$};
    \node (y0) at (-11.5,7)[colorstyle]{$w_1$};
    \node (u2) at (-13,8.5)[colorstyle]{$s$};
    
    \node (x) at (-17,4.5)[colorstyle]{$v$};
    
    \node (u1) at (-10,8)[colorstyle]{$t$};
    
    \node (v1) at (-10,2)[colorstyle]{$z$};
    
    \node (t) at (-11,9)[colorstyle]{$x$};
    
    \node (z) at (-11, 1)[colorstyle]{$y$};
    
    \draw[thick] (y1)--(y0)--(u2)
    (v2)--(yk);
    
    \draw[thick]
    (z)--(v1)
    (z)--(v2)
    (t)--(u1)
    (t)--(u2)
    (x)--(y0)
    (x)--(y1)
    (x)--(yk)
    (x)--(v2)
    (x)--(u2)
    (u1)--(y0)
    (v1)--(yk);
    
    \draw[thick] (x) to [out=110,in=80] (u1);
    \draw[thick] (x) to [out=270,in=300] (v1);
    
    \draw[thick, loosely dotted] (y1)--(yk);

\end{tikzpicture}
\caption{$N_k$, a graph whose complement is an induced-minor-minimal non-$2$-cograph.}
\label{4_deletable_vertices}
\end{figure}

\section{Preliminaries}

Let $G$ be a graph. A vertex $u$ of $G$ is a {\bf neighbour} of a vertex $v$ of $G$ if $uv$ is an edge of $G$.  The {\bf neighbourhood $N_G(v)$} of $v$ in $G$ is the set of all neighbours of $v$ in $G$. %A {\bf $t$-cut} of $G$ is set $X_t$ of vertices of $G$ such that $|X_t| = t$ and $G - X_t$ is disconnected. A graph that has no $t$-cut for all $t$ less than $k$ is {\bf $k$-connected}.
Viewing $G$ as a subgraph of $K_n$ where $n = |V(G)|$, we colour the edges of $G$ green while assigning the colour red to the non-edges of $G$. In this paper, we use the terms {\bf green graph} and {\bf red graph} for $G$ and its complementary graph $\overline{G}$, respectively. An edge of $G$ is called a {\bf green edge} while a {\bf red edge} refers to an edge of $\overline{G}$. The {\bf green degree} of a vertex $v$ of $G$ is the number of {\bf green neighbours} of $v$, while the {\bf red degree} of $v$ is its number of {\bf red neighbours}.

Let $G_1$ and $G_2$ be graphs. If their vertex sets are disjoint, their {\bf $0$-sum}, $G_1 \oplus_0 G_2$, is their disjoint union. Now, suppose that $V(G_1) \cap V(G_2) = T$, that $G_1[T]=G_2[T]$, and that $|T|=t$. Then the union of $G_1$ and $G_2$, which has vertex set $V(G_1) \cup V(G_2)$ and edge set $E(G_1) \cup E(G_2)$, is a {\bf $t$-sum}, $G_1 \oplus_t G_2$, of $G_1$ and $G_2$.

 For $k \geq 1$, a graph $G$ is a {\bf $k$-cograph} if, for every induced subgraph $H$ of $G$, at least one of $H$ and $\overline{H}$ is not $k$-connected. Thus a $1$-cograph is just a cograph. Clearly, every $k$-cograph is also a $(k+1)$-cograph.

We omit the straightforward proofs of the next three results.

\begin{lemma}
\label{closed_induced_subgraph}
Let $G$ be a $k$-cograph. 
\begin{enumerate}[label=(\roman*)]
    \item Every induced subgraph of $G$ is a $k$-cograph.
    \item $\overline{G}$ is a $k$-cograph.
\end{enumerate}
\end{lemma}

\begin{lemma}
\label{closed_direct_sum}
For $0 \leq t < k$, a $t$-sum of two $k$-cographs  is a $k$-cograph.
\end{lemma}

%\begin{proof}
%Let $H$ be an induced subgraph of $G$. If $H$ is an induced subgraph of $G_1$ or $G_2$, we are done. Therefore, we may assume that $H$ intersects both $G - G_1$ and $G- G_2$. This implies $H$ is not $2$-connected and our result follows.
%\end{proof}

\begin{lemma}
\label{complementation_contraction_together}
Let $G$ be a graph and let $uv$ be an edge $e$ of $G$. Then  $\overline{G/e}$ is the graph obtained by adding a vertex $w$ with neighbourhood $N_{\overline{G}}(u) \cap N_{\overline{G}}(v)$ to the graph $\overline{G} - \{u,v\}$.
\end{lemma}

Cographs are also called complement-reducible graphs due to the following recursive-generation result \cite{corneil}. The operation of taking the complement of a graph is called {\bf complementation}.

\begin{lemma}
\label{recursion1}
A graph $G$ is a cograph if and only if $G$ can be generated from $K_1$ using complementation and $0$-sum.
\end{lemma}

Next, we show that, for $k \geq 2$, the class of $k$-cographs can be generated similarly.

\begin{lemma}
\label{recursion2}
For all positive integers $k$, a graph $G$ is a $k$-cograph if and only if $G$ can be generated from $K_1$ using complementation and the operation of $t$-sum for all $t$ with $0 \leq t<k$.
\end{lemma}

\begin{proof}
Let $G$ be a $k$-cograph. If $|V(G)| \leq 2$, the result holds. We proceed via induction on the number of vertices of $G$. Assume that the result holds for all $k$-cographs of order less than $|V(G)|$. Since $G$ is a $k$-cograph, $G$ or $\overline{G}$ is not $k$-connected. Without loss of generality, we may assume that $G$ is not $k$-connected. Therefore, for some $t<k$, we can write $G$ as a $t$-sum of two induced subgraphs $G_1$ and $G_2$ of $G$. By Lemma~ \ref{closed_induced_subgraph}, $G_1$ and $G_2$ are $k$-cographs and the result follows by induction. 

Conversely, let $G$ be a graph that can be generated from $K_1$ using complementation and $t$-sums. Since $K_1$ is a $k$-cograph, the result follows by Lemmas \ref{closed_induced_subgraph} and \ref{closed_direct_sum}.
\end{proof}

The following recursive-generation result for cographs is due to Royle \cite{royle}. It uses the concept of {\bf join} of two disjoint graphs $G$ and $H$, which is the graph $G \bigtriangledown H$ that is obtained from the union of $G$ and $H$ by joining every vertex of $G$ to every vertex of $H$.

\begin{lemma}
\label{royle_lemma}
Let $\mathcal{C}$ be the class of graphs defined as follows:

\begin{enumerate}[label=(\roman*)]
    \item $K_1$ is in $\mathcal{C}$;
    \item if $G$ and $H$ are in $\mathcal{C}$, then so is $G \oplus_0 H$; and
    \item if $G$ and $H$ are in $\mathcal{C}$, then so is $G \bigtriangledown H$.
\end{enumerate}

Then $\mathcal{C}$ is the class of cographs.
\end{lemma}

For graphs $G$ and $H$ such that $V(G) \cap V(H) = T$ and $G[T]=H[T]$, suppose that $|T|=t$. We generalize the join operation letting $G \bigtriangledown_t H$ be the graph that is  obtained from the union of $G$ and $H$ by joining every vertex of $V(G)-V(H)$ to every vertex of $V(H)-V(G)$. Note that $G \bigtriangledown_t H$ is the graph $\overline{\overline{G} \oplus_t \overline{H}}$.

 The next result generalizes Lemma~ \ref{royle_lemma} to $k$-cographs.

\begin{proposition}
\label{royle_lemma_generalize}
For $k \geq 1$, let $\mathcal{C}$ be the class of graphs defined as follows:

\begin{enumerate}[label=(\roman*)]
    \item $K_1$ is in $\mathcal{C}$;
    \item if $G$ and $H$ are in $\mathcal{C}$, then so is $G \oplus_t H$ for all $t$ with $0 \leq t < k$; and
    \item if $G$ and $H$ are in $\mathcal{C}$, then so is $G \bigtriangledown_t H$ for all $t$ with $0 \leq t < k$.
\end{enumerate}

Then $\mathcal{C}$ is the class of $k$-cographs.
\end{proposition}

\begin{proof}
Since $G \bigtriangledown_t H$ can be written in terms of $t$-sum and complementation, every graph in $\mathcal{C}$ is a $k$-cograph. Conversely, let $G$ be a $k$-cograph. If $|V(G)|=1$, then $G \in \mathcal{C}$. We proceed by induction on $|V(G)|$. Let $|V(G)|=n \geq 2$ and assume that $H \in \mathcal{C}$ when $H$ is a $k$-cograph with $|H| < n$. By Lemma~ \ref{recursion2}, $G$ or $\overline{G}$ is a $t$-sum of two smaller $k$-cographs. If $G$ is the graph that can be decomposed as a $t$-sum, then the result follows by induction. Therefore we may assume that $\overline{G}$ is $G_1 \oplus_t G_2$ for two smaller $k$-cographs $G_1$ and $G_2$. Observe that $G = \overline{G_1} \bigtriangledown_t \overline{G_2}$. By Lemma~ \ref{closed_induced_subgraph}, $\overline{G_1}$ and $\overline{G_2}$ are $k$-cographs and so are in $\mathcal{C}$ by induction. Therefore $G$ is in $\mathcal{C}$.
\end{proof}

Next we show that the class of $2$-cographs is closed under contractions.

\begin{lemma}
\label{closed_under_contraction}
Let $G$ be a $2$-cograph and $e$ be an edge of $G$. Then  $G/e$ is a $2$-cograph. 
\end{lemma}

\begin{proof}
Assume to the contrary that $G/e$ is not a $2$-cograph. Then there is an induced subgraph $H$ of $G/e$ such that both $H$ and $\overline{H}$ are $2$-connected. Let $e = uv$ and let $w$ denote the vertex in $G/e$ obtained by identifying $u$ and $v$. We may assume that $w$ is a vertex of $H$, otherwise $H$ is an induced subgraph of $G$, a contradiction. We assert that the subgraph $H'$ of $G$ induced on the vertex set $(V(H) \cup \{u,v\}) - \{w\}$ is $2$-connected, as is its complement $\overline{H'}$. To see this, note that, since $H$ is $2$-connected, $H'$ is $2$-connected unless one of $u$ and $v$, say $u$, is a leaf of $H'$. In the exceptional case, we have $H' - u \cong H$, so $G$ has an induced subgraph for which both it and its complement are $2$-connected, a contradiction. We deduce that $H'$ is $2$-connected. 

By Lemma \ref{complementation_contraction_together}, the neighbours of $w$ in $\overline{H}$ are the common neighbours of $u$ and $v$ in $\overline{H'}$. Thus the degrees of $u$ and $v$ in $\overline{H'}$ each equal at least the degree of $w$ in $\overline{H}$. Moreover, $\overline{H'} - u$ has a spanning subgraph isomorphic to $\overline{H}$ and is therefore $2$-connected. Since $u$ has degree at least two in $\overline{H'}$, it follows that $\overline{H'}$ is $2$-connected, a contradiction.
\end{proof}

We show next that, for all $k \geq 3$, a contraction of a $k$-cograph need not be a $k$-cograph. We use the following construction for the proof. Start with a graph $G$ with vertex set $\{v_1, v_2, \ldots, v_n\}$ and a copy $G'$ of $G$ with vertex set $\{v_1',v_2', \ldots, v_n'\}$. Take the disjoint union of $G$ and $G'$, and add all the edges joining $v_i$ to $v_i'$. The resulting graph, $G \square K_2$, is the Cartesian product of $G$ and $K_2$. 

\begin{lemma}
\label{k_cographs_contract}

For $k \geq 3$, the class of $k$-cographs is not closed under contraction.
\end{lemma}

\begin{proof}
Let $G_2 = C_5$. For all $k \geq 3$, let $G_k = G_{k-1} \square K_2$. One can easily check that $G_k$ is a $k$-connected, $k$-regular graph whose complement is also $k$-connected.

Let $G_k'$ be a graph having an edge $e$ that is in no $3$-cycles such that $G_k'/e = G_k$ and the endpoints of $e$ each have degree less than $k$. Note that every proper induced subgraph of $G_k'$ has a vertex of degree less than $k$ and so $G_k'$ is a $k$-cograph. However, $G_k'/e$ is not a $k$-cograph as it equals $G_k$.
\end{proof}

By Lemmas \ref{closed_induced_subgraph} and \ref{closed_under_contraction}, the class of $2$-cographs is closed under taking induced minors. In the rest of the paper, we will focus our attention on $2$-cographs.
The next lemma is straightforward.

\begin{lemma}
\label{small_cographs}
All graphs having at most four vertices are $2$-cographs. 
\end{lemma}

Since we can compute the blocks of a graph in  polynomial time \cite[4.1.23.]{west}, the algorithm in Figure \ref{identify} recognizes $2$-cographs in polynomial time. Since $2$-cographs do not have induced subgraphs isomorphic to odd cycles of length at least five or their complements, it follows by the Strong Perfect Graph Theorem \cite{perfect} that all $2$-cographs are perfect. However, this inclusion is proper. For example, the graph $C_6^+$ obtained from a $6$-cycle by adding a chord to create two $4$-cycles is a perfect graph that is not a $2$-cograph.

%\begin{algorithm}
%\renewcommand{\thealgorithm}{}
%\caption{Recognizing a $2$-cograph}
\begin{figure}
\begin{algorithmic}[]
\REQUIRE Input a simple graph $G$
\STATE Set $H \leftarrow G$, BlocksList $\leftarrow [G]$

\IF{$|V(H)| \leq 4$}

\STATE remove $H$ from BlocksList

\IF{BlocksList is empty}

\STATE return $G$ is a $2$-cograph and exit the algorithm

\ELSE

\STATE update $H$ to be an element of BlocksList

\ENDIF
\ENDIF

\IF{some $K$ in $\{H,\overline{H}\}$ can be decomposed into $2$-connected blocks}

\STATE remove $H$ from BlocksList
\STATE add all the blocks of $K$ to BlocksList
\STATE update $H$ to be an element of BlocksList

\ELSE

\STATE return $G$ is not a $2$-cograph and exit the algorithm

\ENDIF

\end{algorithmic}
\caption{Algorithm for recognizing a $2$-cograph.}
\label{identify}
\end{figure}

%\end{algorithm}

Akiyama and Harary \cite[Corollary 1a]{harary} claimed that a $2$-connected graph $G$ has a $2$-connected complement if and only if the red and green degrees of every vertex of $G$ are at least two and $G$ has no spanning complete bipartite subgraph. However, this result is not true. The graphs in Figure~\ref{harary_figure} are complements of each other. The first graph in the figure satisfies the hypotheses of \cite[Corollary 1a]{harary} but its complement, $C_4 \oplus_1 C_4$, is not $2$-connected.

\begin{figure}[htbp]
\centering
\begin{minipage}{.30\linewidth}
  \includegraphics[scale=0.30]{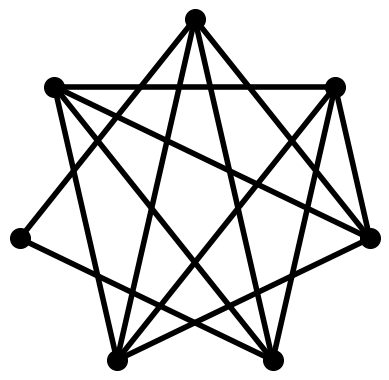}
\end{minipage}
\hspace{.05\linewidth}
\begin{minipage}{.30\linewidth}
  \includegraphics[scale=0.30]{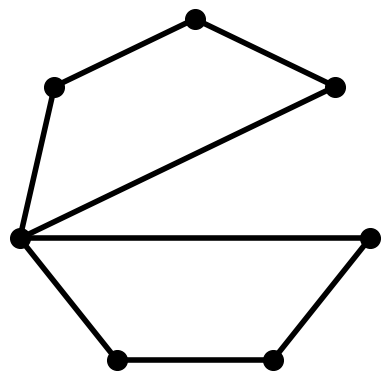}
\end{minipage}

\caption{Counterexample to a result of Akiyama and Harary.}
\label{harary_figure}

\end{figure}

We can repair Akiyama and Harary's result as follows.

\begin{proposition}
\label{corrected_harary}
If $G$ is a $2$-connected graph, then $\overline{G}$ is a $2$-connected graph if and only if $G$ has no complete bipartite subgraph using at least $|V(G)| - 1$ vertices.
\end{proposition}

\begin{proof}
Note that if $\overline{G}$ is not $2$-connected, then $G$ has a spanning complete bipartite subgraph or a complete bipartite subgraph on $|V(G)| - 1$ vertices. The converse is immediate.
\end{proof}

\section{Induced-minor-minimal non-$2$-cographs}

We noted in Section 2 that $2$-cographs are closed under induced minors. In this section, we consider those non-$2$-cographs for which every proper induced minor is a $2$-cograph. We call these graphs {\bf induced-minor-minimal non-$2$-cographs}. The goal of this section is to characterize such graphs. We begin by showing that there are infinitely many of them. Theorem \ref{main_new_addition}, whose proof appears at the end of this section, specifies all of the infinite families of such graphs.

\begin{lemma}
\label{induced_minor_minimal_infinite}
Let $G$ be the complement of a cycle $C$ of length exceeding four. Then $G$ is an induced-minor-minimal non-$2$-cograph. 
\end{lemma}

\begin{proof}

Certainly $G$ is not a $2$-cograph since both $G$ and its complement are $2$-connected. Moreover, by Lemma \ref{closed_induced_subgraph}, $G-v$ is a $2$-cograph for all vertices $v$ of $G$ because $\overline{G}-v$ is a path and is therefore a $2$-cograph. It remains to show that $G/e$ is a $2$-cograph for all edges $e$ of $G$. By Lemma~ \ref{complementation_contraction_together}, the complement of $G/e$ is either a $0$-sum of two paths and an isolated vertex, or a $0$-sum of a path and $K_2$. This implies that the complement of $G/e$ is a $2$-cograph and, by Lemma~ \ref{closed_induced_subgraph}, the result follows.
\end{proof}

Note that the complements of cycles of length at least five are not the only induced-minor-minimal non-$2$-cographs. It can be checked that both $C_6^+$ and its complement are induced-minor-minimal non-$2$-cographs.

The following lemma is obtained by applying \cite[Lemma 2.3]{ox1} (see also  \cite[Lemma 4.3.10]{ox2}) to the bond matroid of a $2$-connected graph.

\begin{lemma}
\label{contractible_edges_vertex}
Let $G$ be a $2$-connected graph other than $K_3$ and let $v$ be an arbitrary vertex of G. Then $G$ has at least two edges $e$ incident to $v$ such that $G/e$ is $2$-connected.
\end{lemma}

An edge $e$ of a $2$-connected graph $G$ is {\bf contractible} if $G/e$ is $2$-connected. The following observation is immediate.

\begin{lemma}
\label{induced_minor_minimal_$2$-connected}
Let $G$ be an induced-minor-minimal non-$2$-cograph. Then both $G$ and $\overline{G}$ are $2$-connected.
\end{lemma}

In the rest of the section, we use the next two theorems of Chan about contractible edges in $2$-connected graphs \cite[Theorems 3.1, 3.3, and 3.5]{chan}. A component of a graph is {\bf trivial} if it has just one vertex. In a $2$-connected graph, a {\bf $2$-cut} is a $2$-element vertex cut.

\begin{theorem}
\label{chan_first_result}
Let $G$ be a $2$-connected graph that is not isomorphic to $K_3$. Suppose all the contractible edges of $G$ meet a $3$-element subset $S$ of $V(G)$. Then either $G-S$ has no edges, or $G-S$ has exactly one non-trivial component and this component has at most three vertices.
\end{theorem}

\begin{theorem}
\label{chan_second_result}
Let $G$ be a $2$-connected graph that is not isomorphic to $K_3$. Suppose all the contractible edges of $G$ meet a subset $S$ of $V(G)$ such that $|S| \geq 4$. Then $G-S$ has at most $|S|-2$ non-trivial components and, between them, these components have at most $2|S|-4$ vertices.
\end{theorem}

We will also frequently use the following straightforward result.

\begin{lemma}
\label{complement_connected_observation}
Let $G$ be a $2$-connected graph. If $G$ has a $2$-cut $\{g_1, g_2\}$ such that each of $g_1$ and $g_2$ has red degree at least two and the components of $G- \{g_1, g_2\}$ can be partitioned into two sets each of which contains at least two vertices, then the red graph $\overline{G}$ is $2$-connected.
\end{lemma}

\begin{lemma}
\label{path_three_basic}
Let $G$ be an induced-minor-minimal non-$2$-cograph such that $|V(G)| \geq 6$ and let $wxyz$ be a path $P$ of $G$ such that both $x$ and $y$ have degree two in $G$. Then $w$ and $z$ are adjacent.
\end{lemma}

\begin{proof}
Assume that $w$ and $z$ are not adjacent. By Lemma~ \ref{induced_minor_minimal_$2$-connected}, $G$ is $2$-connected, so there is a path $P'$ joining $w$ and $z$ such that $P$ and $P'$ are internally disjoint. This implies that $G$ has $C_5$ as a proper induced minor. As $C_5$ is not a $2$-cograph, this is a contradiction.
\end{proof}

\begin{lemma}
\label{path_three}
Let $G$ be an induced-minor-minimal non-$2$-cograph. If $G$ has two adjacent vertices of degree two, then $|V(G)| \leq 10$.
\end{lemma}

\begin{proof}
Assume $|V(G)| \geq 11$. Let $a$ and $b$ be two vertices of $G$ of degree two such that $ab$ is a green edge. Let $c$ be the green neighbour of $a$ distinct from $b$, and let $d$ be the green neighbour of $b$ distinct from $a$. Then $c \neq d$, otherwise $G$ is not $2$-connected, contradicting Lemma~ \ref{induced_minor_minimal_$2$-connected}. By Lemma~ $\ref{path_three_basic}$, $cd$ is a green edge. Observe that every vertex of $V(G) - \{a,b,c,d\}$ has red edges joining it to each of $a$ and $b$. Thus $\overline{G} -\{c,d\}$ is $2$-connected.

Suppose that both $c$ and $d$ have red degree at least three. Let $w$ be a red neighbour of $d$ such that $w \neq a$. It follows by Lemma~ $\ref{contractible_edges_vertex}$ that $w$ has a contractible green edge incident to it, say $e$, such that the other endpoint of $e$ is not $c$. Then $\overline{G/e}$ is $2$-connected, a contradiction.

Next suppose that both $c$ and $d$ have red degree two. First, we assume that $c$ and $d$ have the same red neighbour, say $v$, in $G - \{a,b\}$. Since $v$ has green degree at least two, we have two green neighbours of $v$, say $x$ and $y$. Note that $x$ and $y$ are in $V(G) - \{a,b,c,d\}$. Since $x$ and $y$ are adjacent to both $c$ and $d$ in the green graph, both the red and the green graphs induced on $\{a,b,c,d,v,x,y\}$ are $2$-connected. This implies $|V(G)| \leq 7$, a contradiction. We may now assume that $c$ and $d$ have distinct red neighbours in $G - \{a,b\}$; call them $v$ and $w$, respectively. Note that $vdcw$ is a green $vw$-path.

\begin{sublemma}
\label{disjoint_paths}
$G - \{a,b\}$ has no $vw$-path $P$ internally disjoint from the path $vdcw$.
\end{sublemma}

Assume that $G-\{a,b\}$ has such a path. Observe that the red graph and the green graph induced on the vertex set $V(P) \cup \{a,b,c,d\}$ are $2$-connected and therefore, $V(G) = V(P) \cup \{a,b,c,d\}$. Now $|V(P)| \geq 7$ since $|V(G)| \geq 11$. Let $e$ be an edge in the path $P$ such that neither of the endpoints of $e$ is in $\{v,w\}$. Note that $G/e$ and $\overline{G/e}$ are both $2$-connected, a contradiction. Thus \ref{disjoint_paths} holds.

Let $P_1$ and $P_2$ be shortest $vw$-paths in $G-\{a,b,d\}$ and $G-\{a,b,c\}$ respectively. By \ref{disjoint_paths}, $P_1$ contains the vertex $c$ and $P_2$ contains $d$. Note that $V(G) = V(P_1) \cup V(P_2) \cup \{a,b\}$. As $|V(G)| \geq 11$, we may assume that $P_1 - w$ has length at least three. Let $e$ be an edge in $P_1 - w$ such that the endpoints of $e$ are not in $\{c,v\}$. Note that $G/e$ and $\overline{G/e}$ are both $2$-connected, a contradiction.

Finally, without loss of generality, we may assume that $c$ has red degree two and $d$ has red degree at least three. Let $v$ be the red neighbour of $c$ distinct from $b$. Suppose that $dv$ is red. Let $x$ and $y$ be two green neighbours of $v$ and let $P$ be a shortest path from $d$ to $\{v,x,y\}$ in $G-\{a,b,c\}$. Then, for $V' = \{a,b,c,d,v,x,y\} \cup V(P)$, the red and green graphs induced by $V'$ are $2$-connected, so $V' = V(G)$. As $|V(G)| \geq 11$, we may assume that $P$ has length at least three. Let $e$ be an edge in $P$ such that the endpoints of $e$ are not in $\{d,v,x,y\}$. Note that $G/e$ and $\overline{G/e}$ are both $2$-connected, a contradiction. Therefore, $dv$ is green. Let $w$ be a red neighbour of $d$ in $G-\{a,b\}$. Let $u$ be a green neighbour of $v$ distinct from $d$. Observe that $u \neq w$, otherwise $|V(G)| \leq 6$ since both $G[\{a,b,c,d,v,w\}]$ and $\overline{G}[\{a,b,c,d,v,w\}]$ are $2$-connected. Let $P$ be a shortest path from $w$ to $\{d,u,v\}$ in $G -\{a,b,c\}$. Then $V(G) = \{a,b,c,d,u,v,w\} \cup V(P)$, so we may assume that $P$ has length at least three. Then, for an edge $e$ of $P$ having neither endpoint in $\{d,u,v,w\}$, both $G/e$ and $\overline{G/e}$ are $2$-connected, a contradiction.
\end{proof}

The next lemma shows that if a path of an induced-minor-minimal non-$2$-cograph $G$ has three consecutive vertices of degree two, then $G \cong C_5$.

\begin{lemma}
\label{path_C_5}
Let $G$ be an induced-minor-minimal non-$2$-cograph such that $G$ has a path $P$ of length exceeding three and all the internal vertices of $P$ are of degree two, then $G \cong C_5$.
\end{lemma}

\begin{proof}
Let $u$ and $v$ be vertices of $P$ such that the subpath $P_{uv}$ of $P$ joining $u$ and $v$ has length four. Since $G$ is $2$-connected, there is a $uv$-path $P'$ such that $P_{uv}$ and $P'$ are internally disjoint. Assume that $P'$ is a shortest such path. Then contracting all but one edge in $P'$ and deleting all the vertices not in $V(P_{uv})$, we obtain $C_5$. Since $G$ cannot have $C_5$ as a proper induced minor, $G \cong C_5$.
\end{proof}

A $2$-connected graph $H$ is {\bf critically $2$-connected} if $H -v$ is not $2$-connected for all vertices $v$ of $H$.

\begin{lemma}
\label{two_cases_induced_minor_minimal}
If $G$ is a non-$2$-cograph such that $G-v$ is a $2$-cograph for all vertices $v$ of $G$, then $G$ or $\overline{G}$ is critically $2$-connected, or both $G$ and $\overline{G}$ have vertex connectivity two.
\end{lemma}

\begin{proof}
Certainly, $G$ and $\overline{G}$ are $2$-connected and, for all vertices $v$ of $G$, either $G-v$ or $\overline{G} - v$ is not $2$-connected. Observe that if neither $G$ nor $\overline{G}$ is critically $2$-connected, then $G$ has vertices $v$ and $v_c$ such that $G-v$ and $\overline{G} - v_c$ are $2$-connected. It follows that $G-v_c$ and $\overline{G}-v$ are not $2$-connected so both $G$ and $\overline{G}$ have vertex connectivity two.
\end{proof}

Next we find those induced-minor-minimal non-$2$-cographs $G$ such that $G$ or $\overline{G}$ is critically $2$-connected. We will use the following result of Nebesky \cite{nebesky}.

\begin{lemma}
\label{critical_2_connected}
Let $G$ be a critically $2$-connected graph such that $|V(G)| \geq 6$. Then $G$ has at least two distinct paths of length exceeding two such that the internal vertices of these paths have degree two in $G$.
\end{lemma}

\begin{proposition}
\label{case_1}
Let $G$ be an induced-minor-minimal non-$2$-cograph such that $G$ is critically $2$-connected. Then $G$ is isomorphic to $C_5$ or $C_6^+$.
\end{proposition}

\begin{proof}
By Lemmas \ref{small_cographs} and \ref{induced_minor_minimal_infinite}, it follows that $C_5$ is the unique induced-minor-minimal non-$2$-cograph with at most five vertices, so we may assume that $|V(G)| \geq 6$. Thus, by Lemma~ \ref{critical_2_connected}, $G$ has two distinct paths $P_1$ and $P_2$ of length exceeding two such that their internal vertices have degree two. Since $G$ is not isomorphic to $C_5$, by Lemma~ \ref{path_C_5}, we may assume that both $P_1$ and $P_2$ have length three. Lemma~ \ref{path_three_basic} implies that, for each $i$, the endpoints of $P_i$ are adjacent. We deduce that $G$ has $C_6^+$ as an induced minor. As $C_6^+$ is an induced-minor-minimal non-$2$-cograph, we deduce that $G \cong C_6^+$.
\end{proof}

For a graph $G$, let $V_g$ and $V_r$ be its set of vertices of green-degree two and its set of vertices of red-degree two.

\begin{lemma}
\label{edited_new_case1_james_work}
A graph $G$ is an induced-minor-minimal non-$2$-cograph for which the graph $G[V_r]$ induced on $V_r$ has at least two disjoint red edges if and only if $\overline{G}$ is a cycle with at least five vertices, or $\overline{G}$ is isomorphic to $H_1, H_2$, or $F_k$ for some $k\geq 0$ where $H_1, H_2$, and $F_k$ are shown in Figure $\ref{James_new}$.
\end{lemma}

\begin{proof}
 First we observe that if $\overline{G}$ is a cycle with $|V(\overline{G})| \geq 5$ or if $\overline{G}$ is isomorphic to $H_1, H_2,$ or $F_k$, then $G[V_r]$ has at least two disjoint red edges. Moreover, by Lemma \ref{induced_minor_minimal_infinite}, if $\overline{G}$ is a cycle with $|V(\overline{G})| \geq 5$, then $G$ is an induced-minor-minimal non-$2$-cograph. It is straightforward to check that if $\overline{G}$ is isomorphic to $H_1$ or $H_2$, then $G$ is an induced-minor-minimal non-$2$-cograph. Finally, we show that, for all $k \geq 0$, the complement of $F_k$ is an induced-minor-minimal non-$2$-cograph. Since $F_0 \cong C_6^+$ and the complement of the latter is an induced-minor-minimal non-$2$-cograph, we may assume that $k>0$. As both $F_k$ and $\overline{F_k}$ are $2$-connected, the graph $\overline{F_k}$ is not a $2$-cograph. We show that every proper induced minor $H$ of $\overline{F_k}$ is a $2$-cograph. First assume that $H$ is an induced subgraph of $\overline{F_k}$. Deleting the vertex $x$ from $F_k$ leaves a path, which is a $2$-cograph. Thus we may assume that $x$ is a vertex of $H$. Once a vertex distinct from $x$ is deleted from $F_k$, if we were to find a non-$2$-cograph, it must be contained in one of the blocks of the vertex deletion. Each block $B$ of a vertex deletion of $F_k$ that has at least three vertices must have $x$ as a vertex. Moreover, $B$ has $x$ adjacent to all but at most one other vertex, so its complement is not $2$-connected. It is now straightforward to see that $H$ is a $2$-cograph. For an edge $uv$ of $\overline{F_k}$, it follows by Lemma~ \ref{complementation_contraction_together} that the complement of $\overline{F_k}/uv$ is either an induced subgraph of $F_k$ or a $1$-sum of an induced subgraph of $F_k$ with $K_2$ or $K_3$. Thus $\overline{F_k}/uv$ is a $2$-cograph and so $\overline{F_k}$ is an induced-minor-minimal non-$2$-cograph.

Conversely, assume that $G$ is an induced-minor-minimal non-$2$-cograph for which $G[V_r]$ has $u_1u_2$ and $v_1v_2$ as two disjoint red edges. Since $C_5$ is the unique induced-minor-minimal non-$2$-cograph with five vertices, we may assume that $|V(\overline{G})| \geq 6$, that $\overline{G}$ is not a cycle, and that no $F_k$ for $k \geq 0$ is isomorphic to $\overline{G}$. Next we show the following.

\begin{sublemma}
\label{new_sublemma_edits1}
In $\overline{G}$, no $u_i$ is adjacent to any $v_j$.
\end{sublemma}

Note that if we have a red edge connecting $\{u_1,u_2\}$ to $\{v_1,v_2\}$, then $\overline{G}$ has a path $P$ of length three such that all the vertices of $P$ have red degree two. Let $Q$ be a shortest path in $\overline{G}\ba E(P)$ joining the endpoints of $P$. Then $\overline{G}$ has as an induced subgraph a cycle with edge set $E(P) \cup E(Q)$. This cycle has at least five edges, a contradiction. Thus \ref{new_sublemma_edits1} holds.

 In $\overline{G}$, let $x$ and $y$ be the neighbours of $u_1$ and $u_2$, respectively, other than $u_2$ and $u_1$; and let $w$ and $z$ be the neighbours of $v_1$ and $v_2$, respectively, other than $v_2$ and $v_1$. Because $\overline{G}$ is $2$-connected, it has a cycle $C$ containing $u_1u_2$ and $v_1v_2$. We show next that $C$ is Hamiltonian. Assume it is not. Certainly $\overline{G}[V(C)]$ is $2$-connected. Consider $G[V(C)]$. In it, $u_1$ and $u_2$ are adjacent to every vertex not in $\{x,u_1,u_2,y\}$, and $v_1$ and $v_2$ are adjacent to every vertex not in $\{w,v_1,v_2,z\}$. In addition, $u_1$ and $v_1$ are adjacent to $y$ and it follows by symmetry that $G[V(C)]$ is $2$-connected. The minimality of $G$ implies that $V(G) = V(C)$. Thus $C$ is indeed Hamiltonian.

 Assume that $C$ consists of the path $xu_1u_2y$, a path $P_{yz}$ from $y$ to $z$, the path $zv_2v_1w$, and a path $P_{wx}$ from $w$ to $x$. Now $x$ and $y$ must be distinct. Likewise, $w$ and $z$ are distinct. If $x=w$ and $y=z$, then $\overline{G}$ is either $C_6$ or $C_6^+$. As $C_6^+ = F_0$, this is a contradiction. Thus $x \neq w$ or $y \neq z$.

The graph $\overline{G}-\{u_1,u_2\}$ is connected. Take a shortest path $P$ in this graph from $x$ to $y$. This path $P$ must be a single edge otherwise $\overline{G}$ has an induced cycle of length at least five consisting of the union of $P$ and the path $xu_1u_2y$. By Lemma \ref{induced_minor_minimal_infinite}, the complement of this induced cycle is an induced-minor-minimal non-$2$-cograph, so $G$ is this complement, a contradiction.

By symmetry, we may assume that $\overline{G}$ has $xy$ and $wz$ as edges. Assume that $x=w$ but $y\neq z$. Because the only cycles of $\overline{G}$ containing $u_1u_2$ and $v_1v_2$ are Hamiltonian, the path $P_{yz}$ in $C$ is a shortest path from $y$ to $z$ in $\overline{G}-x$. Let $P_{yz} = y_0y_1 \dots y_k$ where $y=y_0$ and $z=y_k$. For each $i$ in $\{1,2,\ldots, k-1\}$, the only possible neighbour of $y_i$ in $\overline{G}$ other than $y_{i-1}$ and $y_{i+1}$ is $x$. We argue by induction on $i$ that $y_i$ is adjacent to $x$. Suppose $y_1$ is not adjacent to $x$. If $y_2$ is adjacent to $x$, then $\overline{G}$ has $C_6^+$ as an induced subgraph, a contradiction. Thus $y_2$ is not adjacent to $x$. As $y_k$ is adjacent to $x$, for some $j\geq 3$, the vertex $y_j$ is adjacent to $x$, but none of $y_{j-1}, y_{j-2}, \ldots, y_2, y_1$ is adjacent to $x$. Then $\overline{G}$ has a cycle of length at least five as an induced subgraph, a contradiction. We conclude that $y_1$ is adjacent to $x$. Assume that all of $y_1,y_2, \ldots, y_t$ are adjacent to $x$ but $y_{t+1}$ is not. If $y_{t+2}$ is not adjacent to $x$, then $\overline{G}$ contains an induced cycle of length at least five, a contradiction. Thus $y_{t+2}$ is adjacent to $x$ and $\overline{G}$ has $F_t$ as a proper induced subgraph, a contradiction. We conclude that $y_{t+1}$ is adjacent to $x$. Hence, by induction, $y_i$ is adjacent to $x$ for all $i$ in $\{1,2, \ldots, k-1\}$. Thus $\overline{G} \cong F_k$, a contradiction.

It remains to consider the case when $x \neq w$ and $y \neq z$. If $xz$ and $wy$ are both green, then $G-\{v_1,v_2\}$ and its complement are both $2$-connected, a contradiction. Suppose both $xz$ and $wy$ are red. Then $\overline{G}$ has a cycle using $u_1u_2$ and $v_1v_2$ and having exactly eight vertices. Thus $|V(\overline{G})| = 8$. If both $xw$ and $yz$ are green, then $\overline{G}-\{u_1,u_2\} \cong C_6^+$, a contradiction. Thus $\overline{G}$ is isomorphic to either $H_1$ or $H_2$. Now assume that $xz$ is red and $wy$ is green. If both $xw$ and $yz$ are red, then $|V(G)|=8$ and $\overline{G}$ is isomorphic to $H_2$. If $xw$ is green, then, using the paths $P_{yz}$ and $P_{wz}$ in $\overline{G}$, we see that $G-\{v_1,v_2\}$ and its complement are both $2$-connected. Thus we may assume that $xw$ is red. Likewise, $yz$ is red otherwise $G-\{u_1,u_2\}$ and its complement are both $2$-connected, a contradiction. Hence $\overline{G}$ is isomorphic to $H_2$. 
\end{proof}

The following is a straightforward consequence of Lemmas \ref{critical_2_connected} and \ref{edited_new_case1_james_work}.

\begin{proposition}
\label{new_case1_james_work}
A graph $G$ is an induced-minor-minimal non-$2$-cograph for which $\overline{G}$ is critically $2$-connected if and only if $\overline{G}$ is a cycle with at least five vertices, or $\overline{G}$ is isomorphic to $H_1, H_2$, or $F_k$ for some $k\geq 0$.
\end{proposition}

The next three lemmas show that the number of vertices of an induced-minor-minimal non $2$-cograph is bounded above given some conditions on the sizes of components after the removal of a green $2$-cut and on the red degrees of the vertices in that cut.

\begin{lemma}
\label{both_greater}
Let $\{g_1, g_2\}$ be a $2$-cut of an induced-minor-minimal non-$2$-cograph $G$ such that each of $g_1$ and $g_2$ has red degree exceeding two and the components of $G- \{g_1, g_2\}$ can be partitioned into two subgraphs, $A$ and $B$, each having at least two vertices. Then $|V(G)| \leq 8$.
\end{lemma}

\begin{proof}

 Assume that $|V(G)| > 8$. Without loss of generality, let $|V(A)| \geq 4$. Suppose $A$ contains no red neighbours of $g_1$ or $g_2$. Then all vertices in $A$ are incident to both $g_1$ and $g_2$ via a green edge. Let $v$ be any vertex in $A$. Note that both $G-v$ and $\overline{G}-v$ are $2$-connected, a contradiction. Therefore, we may assume that $A$ has a red neighbour, say $a_1$, of $g_1$. Lemma~ \ref{contractible_edges_vertex} implies that we can find a contractible green edge, say $e$, of $G$ incident to $a_1$ such that the other endpoint of $e$ is in $A$. By Lemma~ \ref{complement_connected_observation}, $\overline{G/e}$ is $2$-connected, a contradiction. 
\end{proof}

\begin{lemma}
\label{restructure_lemma}
Let $\{g_1,g_2\}$ be a $2$-cut of an induced-minor-minimal non-$2$-cograph $G$ such that the red degree of $g_1$ is two and the components of $G-\{g_1,g_2\}$ can be partitioned into subgraphs $A$ and $B$ such that $|V(A)| \geq |V(B)| \geq 2$. Suppose that $A$ contains exactly one red neighbour $v$ of $g_1$, and either $g_2$ has no red neighbours in $A-v$, or $g_2$ has red degree greater than two. If all of the contractible edges of $G$ having both endpoints in $V(A) \cup \{g_1,g_2\}$ are incident to a vertex in $\{g_1,g_2,v\}$, then $|V(A)| \leq 4$.
\end{lemma}

\begin{proof}
 Assume that $|V(A)| > 4$.  Let $G_A$ be the subgraph of $G$ induced by $V(A) \cup \{g_1,g_2\}$, and let $Q$ denote the vertex set $\{g_1,g_2,v\}$. By colouring the edge $g_1g_2$ green if necessary, we may assume that $G_A$ is $2$-connected. Since the contractible edges of $G_A$ must meet $Q$, by Theorem $\ref{chan_first_result}$, either $G_A - Q$ has no edges, or $G_A - Q$ has one non-trivial component and this component has at most three vertices. First suppose that $G_A-Q$ is edgeless. Let $\Gamma = V(G_A)-Q$. Next we show the following.

\begin{sublemma}
\label{exactly_one_has_red_degree_two_sub1}
There is no vertex $\gamma$ in $\Gamma$ such that $G_A - \gamma$ is $2$-connected.
\end{sublemma}

If such a vertex exists, then $G - \gamma$ is $2$-connected. Moreover, by Lemma~ \ref{complement_connected_observation}, $\overline{G} - \gamma$ is $2$-connected, a contradiction. Thus \ref{exactly_one_has_red_degree_two_sub1} holds.

\begin{sublemma}
\label{new_sublemma_edits2}
The edge $vg_2$ is red.
\end{sublemma}

Suppose $vg_2$ is green. Let $\alpha$ be a neighbour of $v$ in $\Gamma$. Then $g_1\alpha v g_2g_1$ is a cycle of $G_A$. Because $G_A - Q$ is edgeless and $G_A$ is $2$-connected, every vertex in $\Gamma - \alpha$ is adjacent to at least two members of $\{g_1, g_2, v\}$. Thus $G_A - \gamma$ is $2$-connected for all $\gamma$ in $\Gamma - \alpha$, a contradiction to \ref{exactly_one_has_red_degree_two_sub1}. Thus \ref{new_sublemma_edits2} holds.

Observe that $v$ and $g_2$ have a common neighbour $\beta$ in $\Gamma$ otherwise, as $G_A-Q$ is edgeless, $g_1$ is a cut vertex of $G_A$. By \ref{new_sublemma_edits2}, $v$ has a neighbour $\alpha$ in $\Gamma - \beta$. Since $g_1 \alpha v \beta g_2 g_1$ is a cycle and all vertices in $\Gamma - \{\alpha, \beta\}$ are adjacent to at least two vertices in  $\{g_1, g_2, v\}$, we deduce that $G_A - \gamma$ is $2$-connected for all $\gamma$ in $\Gamma - \{\alpha, \beta\}$, a contradiction.

We may now assume that $G_A - Q$ has one non-trivial component, say $C_A$, and a set $I_A$ of isolated vertices. Moreover, $|V(C_A)| \leq 3$. Then $I_A$ is non-empty since $|V(A)| > 4$. Let $\alpha \beta$ be an edge in $C_A$. Note that $\alpha \beta$ is not contractible in $G_A$,  so $\{\alpha, \beta\}$ is a 2-cut of $G_A$ and, therefore, of $G$. Since $|V(B)| \geq 2$ and $I_A$ is non-empty, each of $\alpha$ and $\beta$ has red degree at least three in $G$. Therefore, by Lemma~ \ref{both_greater}, as $|V(G)| = |V(A)| + 2 + |V(B)| > 8$, there is a vertex $t$ of $G$ whose only green neighbours are $\alpha$ and $\beta$. Since $g_1$ is adjacent to all vertices in $I_A \cup V(C_A)$, it follows that $t= v$. This implies that all vertices in $I_A$ are adjacent only to $g_1$ and $g_2$. Taking $w$ in $I_A$, we see that $G_A-w$ is $2$-connected, a contradiction to \ref{exactly_one_has_red_degree_two_sub1}
\end{proof}

\begin{lemma}
\label{exactly_one_has_red_degree_two}
Let $\{g_1,g_2\}$ be a $2$-cut of an induced-minor-minimal non-$2$-cograph $G$ such that the components of $G-\{g_1, g_2\}$ can be partitioned into subgraphs, $A$ and $B$, each having at least two vertices. If the red degree of $g_1$ is two and that of $g_2$ is greater than two such that one red neighbour of $g_1$ is in $A$ and the other is in $B$, then $|V(G)| \leq 10$.
\end{lemma}

\begin{proof}
Without loss of generality, assume $|V(A)| \geq |V(B)|$. Let $G_A$ be the subgraph of $G$ induced by $V(A) \cup \{g_1, g_2\}$. Note that $G_A$ is $2$-connected since $g_1g_2$ is green. Denote the red neighbour of $g_1$ in $A$ by $v$ and let $Q = \{g_1, g_2, v\}$. Observe that if we have a contractible edge $e$ of $G$ having both endpoints in $V(A) \cup \{g_1,g_2\}$ such that neither of the endpoints of $e$ is in $Q$, then, by Lemma~ \ref{complement_connected_observation}, both $G/e$ and $\overline{G/e}$ are $2$-connected, a contradiction. Therefore, we may assume that all contractible edges of $G$ that have both endpoints in $V(A) \cup \{g_1,g_2\}$ meet $Q$. Thus, by Lemma~\ref{restructure_lemma}, $|V(A)| \leq 4$, so $|V(G)| \leq 10$.
\end{proof}

\begin{lemma}
\label{restructure_a}
Let $\{g_1,g_2\}$ be a $2$-cut of an induced-minor-minimal non-$2$-cograph $G$ such that the components of $G-\{g_1, g_2\}$ can be partitioned into two subgraphs, $A$ and $B$, each having at least two vertices. Suppose that, for each $i$ in $\{1,2\}$, if $g_i$ has red degree two, then $g_i$ has no red neighbour in $B$. Then $|V(B)|=2$.

\end{lemma}

\begin{proof}
Suppose $|V(B)| \geq 3$. If all vertices in $B$ are green neighbours of both $g_1$ and $g_2$, then $G-z$  is $2$-connected for all $z$ in $V(B)$. But, by Lemma~ \ref{complement_connected_observation}, $\overline{G}-z$ is also $2$-connected, a contradiction. Thus $B$ has a red neighbour, say $b$, of $g_1$. Note that $g_1$ has red degree greater than two. Now, by Lemma~\ref{contractible_edges_vertex}, we can find a contractible edge, say $e$, of $G$ incident to $b$ such that the other endpoint of $e$ is in $V(B)$. By Lemma~\ref{complement_connected_observation}, $\overline{G/e}$ is $2$-connected, a contradiction.
\end{proof}

\begin{lemma}
\label{restructure_b}
Let $\{g_1,g_2\}$ be a $2$-cut of an induced-minor-minimal non-$2$-cograph $G$ such that the red degree of $g_1$ is two and the components of $G-\{g_1,g_2\}$ can be partitioned into subgraphs $A$ and $B$ such that $|V(A)| \geq |V(B)| \geq 2$. Suppose that one of the following holds.

\begin{enumerate}[label=(\roman*)]
    \item $A$ contains both the red neighbours $\{x,y\}$ of $g_1$, and $g_2$ has no red neighbour in $A-\{x,y\}$ if the red degree of $g_2$ is two; or
    
    \item $g_2$ has red degree two and $A$ contains exactly one pair $\{x,y\}$ of distinct vertices such that $x$ is a red neighbour of $g_1$, and $y$ is a red neighbour of $g_2$.
\end{enumerate}

 \noindent If all contractible edges of $G$ having both endpoints in $V(A) \cup \{g_1,g_2\}$ are incident to a vertex in $\{g_1,g_2,x,y\}$, then $|V(A)| \leq 6$.

\end{lemma}

\begin{proof}
Assume that $|V(A)| > 6$ and so $|V(G)| > 10$. Let $G_A$ be the graph induced on $V(A) \cup \{g_1,g_2\}$. Let $Q = \{g_1, g_2, x, y\}$. By colouring the edge $g_1g_2$ green if necessary, we may assume that $G_A$ is $2$-connected. Note that all the contractible edges of $G_A$ must meet $Q$, otherwise we have a contractible edge $e$ of $G$ such that $\overline{G/e}$ is $2$-connected, a contradiction. By Theorem $\ref{chan_second_result}$, $G_A - Q$ has at most two non-trivial components and, between them, these components have at most four vertices. 

Let $I_A$ and $N_A$ be the sets of isolated and non-isolated vertices of $G_A-Q$ respectively. We note the following.

\begin{sublemma}
\label{same_nbd_sublemma}
If two vertices $i_1$ and $i_2$ in $I_A$ have the same green neighbourhood in $G$, then $\{i_1, i_2\}$ is a green $2$-cut in $G$.
\end{sublemma}

As $\overline{G} - \{g_1,g_2,i_1\}$ is a complete bipartite graph with each part having at least two vertices, it is $2$-connected. Both $g_1$ and $g_2$ have at least two red neighbours in $\overline{G} - \{i_1\}$. Thus $\overline{G} - i_1$ is $2$-connected. Therefore $G-i_1$ is not $2$-connected. It follows that $i_2$ is a cut-vertex of $G-i_1$ and so $\{i_1, i_2\}$ is a green $2$-cut. Thus \ref{same_nbd_sublemma} holds.

 First suppose that $N_A$ is empty. As $|V(A)| \geq 7$, we see that $|I_A| \geq 5$. Suppose that $g_2$ has red degree two. Then all vertices in $I_A$ are adjacent to both $g_1$ and $g_2$ in $G$. Observe that if a vertex $s$ in $I_A$ has green neighbourhood $\{g_1,g_2\}$, then both $G-s$ and $\overline{G}-s$ are $2$-connected, a contradiction. Since $g_1$ and $g_2$ have no red neighbours in $I_A$, the green neighbourhood of a vertex in $I_A$ is $\{g_1,g_2,x\}, \{g_1,g_2,y\},$ or $\{g_1,g_2,x,y\}$. It follows that there are at least two pairs of vertices in $I_A$ such that each vertex in a pair has the same green neighbourhood. Let $\{i_1,i_2\}$ be such a pair. By \ref{same_nbd_sublemma}, $\{i_1, i_2\}$ is a green $2$-cut. Since the red degrees of both $i_1$ and $i_2$ are greater than two, by Lemma \ref{both_greater}, it follows that there is a vertex $t$ of $G$ that has green neighbourhood $\{i_1, i_2\}$. Note that $t$ is either $x$ or $y$. Since we have at least two such green $2$-cuts, it follows that $g_2x$ and $g_2y$ are both red, and there is a red edge connecting $\{x,y\}$ to $I_A$. Observe that $B$ has no red neighbour of $g_1$ or $g_2$. It now follows that, for each $b$ in $V(B)$, both $G-b$ and $\overline{G}-b$ are $2$-connected, a contradiction. Therefore $g_2$ has red degree at least three. By Lemma \ref{restructure_a}, $|V(B)|=2$. Suppose there is no red edge connecting $\{x,y\}$ to $I_A$. Then the possible green neighbourhoods of the vertices in $I_A$ are $\{x,y\}, \{x,y,g_1\}, \{x,y,g_2\},$ or $\{x,y,g_1,g_2\}$. Thus, by \ref{same_nbd_sublemma}, $I_A$ contains a green $2$-cut $\{i_1,i_2\}$ of $G$. Then we get $|V(G)| \leq 8$ by applying Lemma~ \ref{both_greater} to the green $2$-cut $\{i_1,i_2\}$. Therefore there is a red edge connecting $\{x,y\}$ to $I_A$. It follows that, for some $b$ in $V(B)$, both $G-b$ and $\overline{G}-b$ are $2$-connected, a contradiction.

 We may now assume that $G_A - Q$ has at least one non-trivial component. Let $C$ be such a  component and let $\alpha \beta$ be an edge in $C$. Since $\alpha \beta$ is a non-contractible edge of $G_A$, we see that $\{\alpha, \beta\}$ is a green 2-cut of $G_A$ and thus of $G$. Then $G_A - Q \neq C$ otherwise, by Theorem \ref{chan_second_result}, $|V(A)| \leq 6$, a contradiction. Thus both $\alpha$ and $\beta$ have red degree at least three in $G$. Therefore, by Lemma~\ref{both_greater}, $G$ has a unique vertex $t$ that has green neighbourhood $\{\alpha, \beta\}$. Since all vertices in $G_A$ except $x$ and $y$ are adjacent to $g_1$ via a green edge, $t$ is either $x$ or $y$. As $\alpha \beta$ is an arbitrary green edge in $G_A - Q$, it follows that $G_A - Q$ has at most two edges and therefore has either one non-trivial component with at most three vertices, or has two non-trivial components each with two vertices.

 Suppose that $G_A - Q$ has only one edge, $\alpha \beta$, and let $t$ be the unique member of $\{x,y\}$ that has green neighbourhood $\{\alpha, \beta \}$. Then $|I_A| \geq 3$ and the green neighbourhood of every vertex in $I_A$ is contained in $\{g_1,g_2,s\}$ where $\{t,s\} = \{x,y\}$. It is clear that if a vertex $w$ in $I_A$ has green neighbourhood $\{g_1,g_2\}$, then $G-w$ and $\overline{G}-w$ are $2$-connected. It follows that the green neighbourhood of a vertex in $I_A$ is either $\{g_1,s\}$ or $\{g_1,g_2,s\}$. As $|I_A| \geq 3$, it contains vertices $i_1$ and $i_2$ that have the same green neighbourhood. By \ref{same_nbd_sublemma}, $\{i_1,i_2\}$ is a green $2$-cut in $G$. As neither $t$ nor $s$ has $\{i_1,i_2\}$ as its green neighbourhood, Lemma \ref{both_greater} gives the contradiction that $|V(G)| \leq 8$. We now know that $G_A-Q$ has exactly two edges, so $3 \leq |N_A| \leq 4$. Observe that $I_A \neq \emptyset$ and all vertices in $I_A$ have green neighbourhood equal to $\{g_1,g_2\}$ since $x$ and $y$ have their green neighbourhoods contained in $N_A$. Thus, for $w \in I_A$, both $G-w$ and $\overline{G}-w$ are $2$-connected, a contradiction.
\end{proof}

\begin{lemma}
\label{contrasting_colours}
Let $\{g_1,g_2\}$ be a $2$-cut of an induced-minor-minimal non-$2$-cograph $G$ such that $g_1$ and $g_2$ are not adjacent in $G$ and the components of $G-\{g_1, g_2\}$ can be partitioned into two subgraphs, $A$ and $B$, each having at least two vertices. Then $|V(G)| \leq 10$.
\end{lemma}

\begin{proof}
First suppose that both $g_1$ and $g_2$ have red degree two and that the red neighbour $v$ of $g_1$ that is distinct from $g_2$ is in $A$, and the red neighbour $u$ of $g_2$ distinct from $g_1$ is in $B$. We may assume that $|V(A)| \geq |V(B)|$. Observe that, if we can find a contractible edge $e$ of $G$ having both the endpoints in $V(A) - v$, then, by Lemma~\ref{complement_connected_observation}, $\overline{G/e}$ is $2$-connected, a contradiction. This implies that all the contractible edges of $G$ that have both endpoints in $V(A) \cup \{g_1,g_2\}$ are incident to $\{g_1,g_2,v\}$. By Lemma \ref{restructure_lemma}, $|V(A)| \leq 4$ and so $|V(G)| \leq 10$. Thus we may assume that both $u$ and $v$ are in $A$ and all contractible edges of $G$ that have both endpoints in $V(A) \cup \{g_1,g_2\}$ are incident to $\{g_1,g_2,u,v\}$. Note that $u \neq v$, otherwise $\overline{G}$ has a cut vertex. We get our result now by Lemmas \ref{restructure_a} and \ref{restructure_b}. We may now assume that the red degree of $g_2$ exceeds two.
By Lemma~\ref{both_greater}, we may further assume that the red degree of $g_1$ is two. 

Let $v$ be the red neighbour of $g_1$ other than $g_2$. We may assume that $v$ is in $A$. By Lemma \ref{restructure_a}, $|V(B)|=2$. Note that all the contractible edges of $G$ that have both endpoints in $V(A) \cup \{g_1,g_2\}$ are incident to $\{g_1,g_2,v\}$. The result now follows by Lemma~\ref{restructure_lemma}.
\end{proof}

Lemma \ref{both_greater} can be modified as follows.

\begin{proposition}
\label{bound_10_first}
Let $\{g_1, g_2\}$ be a $2$-cut of an induced-minor-minimal non-$2$-cograph $G$ such that the components of $G-\{g_1, g_2\}$ can be partitioned into two subgraphs, $A$ and $B$, each having at least two vertices. If $g_2$ has red degree greater than two, then $|V(G)| \leq 10$.
\end{proposition}

\begin{proof}
Assume that $|V(G)| \geq 11$. Then, by Lemma \ref{both_greater}, the red degree of $g_1$ is two.
Let $x$ and $y$ be the two red neighbours of $g_1$. Note that if $x$ is in $A$ and $y$ is in $B$, then the result follows by Lemma \ref{exactly_one_has_red_degree_two}. By Lemma \ref{contrasting_colours}, we may suppose that the edge $g_1g_2$ is green and both $x$ and $y$ are in $A$. 

The graph $G_A$ induced on $V(A) \cup \{g_1,g_2\}$ is $2$-connected. Let $Q = \{g_1, g_2, x, y\}$. Then every contractible edge $e$ of $G_A$ must meet $Q$ otherwise, by Lemma~ \ref{complement_connected_observation}, we obtain the contradiction that both $G/e$ and $\overline{G/e}$ are $2$-connected. The result now follows by Lemmas \ref{restructure_a} and \ref{restructure_b}.

\end{proof}

We can generalize the above result by removing the condition on the red degrees of the vertices in the $2$-cut at the cost of raising the bound on the number of vertices of $G$ to $16$.

\begin{proposition}
\label{gen_prop_bound_30}
Let $\{g_1, g_2\}$ be a $2$-cut of an induced-minor-minimal non-$2$-cograph $G$ such that the components of $G-\{g_1, g_2\}$ can be partitioned into two subgraphs, $A$ and $B$, each having at least two vertices. Then $|V(G)| \leq 16$.
\end{proposition}

\begin{proof}
Assume that $|V(G)| \geq 17$. By Lemma \ref{contrasting_colours} and Proposition \ref{bound_10_first}, we may assume that the red degrees of both $g_1$ and $g_2$ are two and $g_1g_2$ is green. We may further assume that $|V(A)| \geq |V(B)|$. The graph $G_A$ induced on $V(A) \cup \{g_1,g_2\}$ is $2$-connected.  Let $Q$ be the union of $\{g_1,g_2\}$, the set of red neighbours of $g_1$ in $A$, and the set of red neighbours of $g_2$ in $A$. Then every contractible edge $e$ of $G_A$ must meet $Q$, otherwise, by Lemma \ref{complement_connected_observation}, we obtain a contradiction. Note that if $|Q|=2$, then, by Lemma \ref{restructure_a}, $|V(A)|=2$ and so $|V(G)| \leq 6$, a contradiction. By Theorems \ref{chan_first_result} and \ref{chan_second_result}, $G_A-Q$ has at most four non-trivial components and between them, these components have at most eight vertices. 

Let $I_A$ and $N_A$ be the sets of isolated and non-isolated vertices of $G_A-Q$, respectively. We note the following.

\begin{sublemma}
\label{bound_N_A_new}
$|N_A| \leq 4$.
\end{sublemma}

Assume that $|N_A| > 4$ and so $G_A-Q$ has at least three edges. Let $\alpha\beta$ be an edge of $G_A-Q$. Because $\alpha \beta$ is not a contractible edge of $G_A$, it follows that $\{\alpha, \beta\}$ is a green $2$-cut of $G$. Observe that each of $\alpha$ and $\beta$ has red degree at least three unless $|I_A|$ is empty, and $G_A - Q$ has one non-trivial component, and $|V(B)|=2$. The exceptional case does not arise since it implies, as $V(G)= (V(A)-Q) \cup Q \cup V(B)$, that $|V(G)| \leq 8+6+2= 16$, a contradiction. By Lemma \ref{both_greater}, there is a vertex $t$ that has green neighbourhood $\{\alpha, \beta\}$. Note that the only vertices that could have green neighbourhood $\{\alpha, \beta\}$ are the common red neighbours of $g_1$ and $g_2$. Since there are at most two such vertices and at least three edges in $G_A-Q$, each of which must have an associated such vertex, we have a contradiction. Thus \ref{bound_N_A_new} holds.

%As in proof of Lemma \ref{restructure_b}, we make note of the following useful observation.

%\begin{sublemma}
%\label{same_nbd_sublemma_new}
%If two vertices $i_1$ and $i_2$ in $I_A$ have the same green neighbourhood, then $\{i_1, i_2\}$ is a green $2$-cut.
%\end{sublemma}

Next we show the following.

\begin{sublemma}
\label{bound_I_A_new}
$|I_A| \leq 4$.
\end{sublemma}

Assume that $|I_A| \geq 5$. Note that all vertices in $I_A$ are adjacent to both $g_1$ and $g_2$. Suppose $I_A$ contains a vertex $i$ such that all vertices in $Q - \{ g_1,g_2 \}$ have degree at least two in $G - i$. Then both $G - i$ and $\overline{G} – i$ are $2$-connected, a contradiction. It follows that, for every vertex $i$ of $I_A$, there is a \textit{special} green edge joining $i$ to a vertex $q$ of $Q-\{g_1,g_2\}$ such that $q$ has green degree two. The set $Q'$ of such vertices $q$ is contained in $Q-\{g_1,g_2\}$. If a member $q'$ of $Q'$ is a common red neighbour of $g_1$ and $g_2$, then it meets at most two special green edges from $I_A$. If, instead, $q'$ has a single red neighbour in $\{g_1,g_2\}$, then it has a single green neighbour in $\{g_1,g_2\}$ and so meets at most one special green edge. Thus the number of red edges from $\{g_1,g_2\}$ to $Q'$ is an upper bound on the number of special green edges from $I_A$. Hence $|I_A| \leq 4$, a contradiction. Thus \ref{bound_I_A_new} holds.

\begin{sublemma}
\label{sublemma_new_edits3}
$|V(B)| \geq 3$.
\end{sublemma}

Suppose that $|V(B)|=2$. Then, by \ref{bound_N_A_new} and \ref{bound_I_A_new}, $|V(G)| \leq 4+4+6+2=16$, a contradiction. Thus \ref{sublemma_new_edits3} holds.

By \ref{sublemma_new_edits3}, since $|V(B)| \neq 2$, Lemma \ref{restructure_a} implies that $B$ contains at least one red neighbour of $\{g_1,g_2\}$. Assume that $B$ contains exactly one such red neighbour $v$. Let $x$ and $y$ be two green neighbours of $v$ in $V(B) \cup \{g_1,g_2\}$. If $V(B)-\{v,x,y\}$ contains a vertex $t$, then $G-t$ and $\overline{G}-t$ are both $2$-connected. It follows that $|V(B)|\leq 3$. Again by \ref{bound_N_A_new} and \ref{bound_I_A_new}, we get $|V(G)| \leq 4+4+5+3=16$, a contradiction. Note that if $A$ contains exactly one of the red neighbours of $\{g_1,g_2\}$, then, by Lemma \ref{restructure_lemma}, $|V(A)| \leq 4$, so $|V(G)| \leq 10$, a contradiction. We may now assume that the red neighbourhood of $\{g_1,g_2\}$ has size four, and each of $A$ and $B$ contains exactly two of those vertices. Then, by Lemma \ref{restructure_b}, $|V(A)| \leq 6$, so $|V(G)| \leq 14$, a contradiction.
\end{proof}

The following corollary summarizes our results about the induced-minor-minimal non-$2$-cographs so far.

\begin{corollary}
\label{new_corollary}
Let $G$ be an induced-minor-minimal non-$2$-cograph. Then

\begin{enumerate}[label=(\roman*)]
    \item $|V(G)| \leq 16;$ or
    \item $\overline{G}$ is a cycle of length at least five; or
    \item $G$ has vertex connectivity two, and, for every $2$-cut $\{g_1,g_2\}$ of $G$, the graph $G-\{g_1,g_2\}$ has exactly two components and one component contains a single vertex.
\end{enumerate}
\end{corollary}

If an induced-minor-minimal non-$2$-cograph $G$ satisfies $(iii)$ of the above corollary, we say that $G$ is an induced-minor-minimal non-$2$-cograph of \textit{type $(iii)$}. The next lemma identifies several infinite families of such graphs.

\begin{lemma}
\label{new_graphs}
Let $G$ be a graph such that $\overline{G}$ is isomorphic to $L_k, M_k, M_k', N_k,$ $N_k'$, or $N_k''$ for some $k\geq 1$ where $L_k,M_k,$ and $N_k$ are shown in Figures $\ref{2_deletable_vertices}, \ref{3_deletable_vertices}$, and $\ref{4_deletable_vertices}$, respectively. Then $G$ is an induced-minor-minimal non-$2$-cograph of type $(iii)$.
\end{lemma}

\begin{proof}
 It is clear that $G$ is not a $2$-cograph as both $G$ and $\overline{G}$ are $2$-connected. Assume that $H$ is an induced subgraph of $G$ such that both $H$ and $\overline{H}$ are $2$-connected. It is clear that $v \in V(H)$ otherwise $H$ or $\overline{H}$ is not $2$-connected. Note that $V(H)$ also contains the vertices $x$ and $y$ since $x$ and $y$ are the only green neighbours of $v$. It now follows that $V(H)$ contains the red neighbours of $x$ and the red neighbours of $y$. It is now straightforward to see that $H = G$. Therefore every proper induced subgraph of $G$ and of  $\overline{G}$ is a $2$-cograph. For an edge $\alpha \beta$ of $G$, it follows by Lemma \ref{complementation_contraction_together} that the complement of $G/\alpha \beta$ is either a proper induced subgraph of $\overline{G}$ or a proper induced subgraph of $\overline{G}$ $1$-summed with $K_2$ or $K_3$. Thus $G/\alpha \beta$ is a $2$-cograph and so $G$ is an induced minor-minimal non-$2$-cograph. Moreover, $\{x,y\}$ is its unique $2$-cut and $G$ is an induced-minor-minimal non-$2$-cograph of type $(iii)$.
\end{proof}

By a similar argument to that just given, we obtain the following.

\begin{lemma}
\label{new_lemma_edits4}
For a non-negative integer $k$, the graph $\overline{F_k}$ is an induced-minor-minimal non-$2$-cograph of type $(iii)$.
\end{lemma}

 In the rest of the section, we find all the other classes of induced-minor-minimal non-$2$-cographs of type $(iii)$ thereby proving Theorem \ref{main_new_addition}.

\begin{lemma}
\label{first_new_lemma}
Let $G$ be an induced-minor-minimal non-$2$-cograph of type $(iii)$. Then $|V_g| \leq 3$ or $\overline{G}$ is of type $(iii)$.
\end{lemma}

\begin{proof}
Suppose that $\overline{G}$ is not of type $(iii)$. By Lemma \ref{two_cases_induced_minor_minimal} and Proposition \ref{case_1}, we may assume that $\overline{G}$ has vertex connectivity two. Take a red $2$-cut $\{r_1,r_2\}$ of $G$ such that the components of $\overline{G}-\{r_1,r_2\}$ can be partitioned into subgraphs $A$ and $B$, and $|V(A)| \geq |V(B)| \geq 2$. If $|V(B)| \geq 3$, then all vertices in $V(G)-\{r_1,r_2\}$ have green degree at least three and so $|V_g| \leq 2$. Now suppose that $V(B) = \{b_1,b_2\}$. Note that there is at most one vertex $a$ in $A$ that has green neighbourhood $\{b_1,b_2\}$ since $G$ is of type $(iii)$. One can now check that all vertices in $V(G)-\{r_1,r_2,a\}$ have green degree at least three, and so $|V_g| \leq 3$.
\end{proof}

\begin{lemma}
\label{second_new_lemma}
Let $G$ be an induced-minor-minimal non-$2$-cograph such that $|V(G)| > 10$. Suppose that $\overline{G}$ is not isomorphic to a cycle or to $F_k$ for some $k \geq 0$. Then the graph induced on the vertex set $V_g$ is a complete red graph and the graph induced on $V_r$ has at most one red edge.
\end{lemma}

\begin{proof}
By Lemma \ref{path_three}, the graph induced on $V_g$ is a complete red graph. Assume that the graph induced on $V_r$ has two red edges $e=u_1u_2$ and $f=v_1v_2$. Note that if $e$ and $f$ are disjoint, then, by Lemma \ref{edited_new_case1_james_work}, we obtain a contradiction. Therefore we may assume that $u_2=v_1$. Let $\alpha$ and $\beta$ be the respective neighbours of $u_1$ and $v_2$ in $\overline{G}-v_1$. Note that $\alpha$ and $\beta$ are distinct otherwise we have a cut vertex in $\overline{G}$, a contradiction. Let $P$ be a shortest $\alpha\beta$-path distinct from $\alpha u_1u_2v_2 \beta$. Then $P$ avoids $\{u_1,u_2,v_2\}$ and the red graph induced on $V(P) \cup \{u_1,u_2,v_2\}$ is a cycle. It follows by the minimality of $G$ that $V(G) = V(P) \cup \{u_1,u_2,v_2\}$ and so $\overline{G}$ is a cycle, a contradiction.

\end{proof}

In the following lemma, we note that either $|V_g|$ or $|V_r|$ is bounded.

\begin{lemma}
\label{V_g_or_V_r_has_bound_four}
Let $G$ be an induced-minor-minimal non-$2$-cograph. Then either $|V_g|$ or $|V_r|$ is at most three, or $|V_g| = |V_r| =4$.
\end{lemma}

\begin{proof}

Note that there are at most $2|V_g|$ green edges and at most $2|V_r|$ red edges joining a vertex in $V_g$ to a vertex in $V_r$. Since there are $|V_g||V_r|$ edges joining vertices in $V_g$ to vertices in $V_r$, we have

$$2|V_g| + 2|V_r| \geq |V_g||V_r|.$$

This inequality is symmetric with respect to $|V_g|$ and $|V_r|$, so we may assume that $|V_g| \geq |V_r|$. Then $2 + 2 \frac{|V_r|}{V_g} \geq |V_r|$. Thus $|V_r| \leq 4$. Moreover, if $|V_r|=4$, then $|V_g|=4$.
\end{proof}

Next we note the following useful observation.

\begin{lemma}
\label{vertices_outside_V_g_union_V_r_second}
Let $G$ be an induced-minor-minimal non-$2$-cograph such that $|V(G)| > 10$. If all vertices of a subset $S$ of $V(G)-(V_g \cup V_r)$ either have a red neighbour in $V_r$ or a green neighbour in $V_g$, then 

\begin{sublemma}
\label{bound_S_new}
$|S| \leq 2|V_g\cup V_r| - |V_g||V_r|.$
\end{sublemma}

 \noindent Moreover, when equality holds here, either each vertex in $S$ has exactly one green neighbour in $V_g$ or has exactly one red neighbour in $V_r$ but not both. In particular, if $S = V(G)-(V_g \cup V_r)$, then 
$$11 + |V_g||V_r| \leq 3|V_g| + 3|V_r|.$$
\end{lemma}

\begin{proof}
There are $|V_g||V_r|$ red or green edges joining a vertex in $V_g$ to a vertex in $V_r$. There are at most $2|V_g|$ green such edges and at most $2|V_r|$ red such edges. Thus among the green edges meeting $V_g$ and the red edges meeting $V_r$ at most $2|V_g \cup V_r| - |V_g||V_r|$ have an endpoint in $V(G) - (V_g \cup V_r)$. Therefore, $|S| \leq 2|V_g\cup V_r| - |V_g||V_r|$ and it is clear that, when equality holds, each vertex in $S$ satisfies the given condition. If $S = V(G)- (V_g \cup V_r)$, then it is clear that $11 + |V_g||V_r| \leq 3|V_g| + 3|V_r|$ since $|V(G)| \geq 11$.
\end{proof}

Lemma \ref{first_new_lemma} can be improved in the following way.

\begin{lemma}
\label{first_new_lemma_improved}
Let $G$ be an induced-minor-minimal non-$2$-cograph of type $(iii)$. Then $|V(G)| \leq 10$ or $|V_g| \leq 3$.
\end{lemma}

\begin{proof}
 By Lemma \ref{first_new_lemma}, it is enough to show that if $\overline{G}$ is of type $(iii)$, then $|V(G)| \leq 10$ or $|V_g| \leq 3$. Suppose that $\overline{G}$ is of type $(iii)$. Since every vertex of $V(G)$ is either in a red $2$-cut or a green $2$-cut, and both $G$ and $\overline{G}$ are of type $(iii)$, we have the following.

\begin{sublemma}
\label{new_sublemma_property_a}
Every vertex in $V(G)$ either has a green neighbour in $V_g$ or a red neighbour in $V_r$.
\end{sublemma}

Since a vertex in $V_g$ has no green neighbour in $V_g$ by Lemma \ref{path_three}, it follows by \ref{new_sublemma_property_a} that $|V_g| \leq 2|V_r|$ since the number o f red-degree-two neighbours of vertices in $V_g$ is at least $|V_g|$ and at most $2|V_r|$.
The following is an immediate consequence of Lemma \ref{vertices_outside_V_g_union_V_r_second} and \ref{new_sublemma_property_a}.

\begin{sublemma}
\label{new_sublemma_counting_vertices}
$|V(G)| \leq |V_g|+|V_r| + 2|V_r \cup V_g| - |V_g||V_r| = 3|V_g|+3|V_r|-|V_g||V_r|$.

\end{sublemma}

Note that if $|V_g| = |V_r| = 4$, then $|V(G)| \leq 8$ and the result holds. Therefore, by Lemma \ref{V_g_or_V_r_has_bound_four}, we may assume that $|V_r|$ is at most three. 
 As $|V_g| \leq 2|V_r|$, by \ref{new_sublemma_counting_vertices}, checking the possibilities for $|V_r|$, we obtain that $|V(G)| \leq 10$.
\end{proof}

In the next proof, we adopt the convention that, for a $2$-cut $\{x,y\}$ of a graph $H$, the graphs $A$ and $B$ are disjoint subgraphs of $H-\{x,y\}$ with $V(A) \cup V(B) = V(H - \{x,y\})$ such that $|V(A)| \geq |V(B)|$, and $|V(B)|$ is maximal. 

\begin{lemma}
\label{first_new_lemma_more_improved}
Let $G$ be an induced-minor-minimal non-$2$-cograph such that $G$ is of type $(iii)$. Then  $|V(G)|\leq 16$ or $|V_g| \leq 1$.
\end{lemma}

\begin{proof}

By Lemma \ref{first_new_lemma_improved}, we may assume that $|V_g| \leq 3$. The following observation is immediate.

\begin{sublemma}
\label{new_sublemma_edits5}
Let $\{r_1,r_2\}$ be a red $2$-cut of $G$. If $|V(B)| \geq 3$, then $\{r_1,r_2\} \subseteq V_g$.
\end{sublemma}

Next we show the following.

\begin{sublemma}
\label{new_sublemma_edits6}
There are at most two vertices outside of $V_g$ that have neither a red neighbour in $V_r$ nor a green neighbour in $V_g$.
\end{sublemma}

Every vertex $v$ of $V(G)-V_g$ is in a green $2$-cut or a red $2$-cut. In the first case, because $G$ is of type $(iii)$, $v$ has a green neighbour in $V_g$. In the second case, let $\{v,r\}$ be a red $2$-cut. By \ref{new_sublemma_edits5}, we may assume that $|V(B)| \leq 2$. If $|V(B)|=1$, then $v$ has a red neighbour in $V_r$. Suppose $|V(B)|=2$. If $w$ is a vertex with green neighbourhood $V(B)$, then $V(B)$ is a green $2$-cut. As $G$ is of type $(iii)$, $w$ is unique. If $|V_g|=3$, it follows that $\{v,r\} \subseteq V_g$, a contradiction, so \ref{new_sublemma_edits6} holds.

If $|V_g| \leq 1$, then the lemma holds, so we may assume $|V_g|=2$. For the red $2$-cut $\{v,r\}$, we know that $|V(B)|=2$. Now each vertex $u$ of $V(G)-V(B)-\{v,r\}$ has $V(B)$ in its green neighbourhood. Thus $\{u,r\}$ cannot be a red $2$-cut with the same $V(B)$. Thus $\{v,r\}$ is the unique red $2$-cut with the given $V(B)$. As $v \notin V_g$ and $G$ is of type $(iii)$, the set $V(B)$ is the green neighbourhood of exactly one vertex in $V_g$. Since $|V_g|=2$, it follows that we have at most two red $2$-cuts for which $|V(B)|=2$. Moreover, each such red $2$-cut contains a member of $V_g$. Now \ref{new_sublemma_edits6} follows immediately. 

By \ref{new_sublemma_edits6} and Lemma \ref{vertices_outside_V_g_union_V_r_second}, $|V(G)| \leq 3|V_g| + 3|V_r| - |V_g||V_r| + 2$. If $|V_g|=3$, then $|V(G)| \leq 11$. Suppose $|V_g|=2$. Then, by Lemma \ref{second_new_lemma}, $|V_r| \leq 2|V_g| + 2 + 2 = 8$, so $|V(G)| \leq 16$.
\end{proof}

\begin{proof}[Proof of Theorem \ref{main_new_addition}]
\setcounter{theorem}{32}
\setcounter{sublemma}{0}
We may assume that $G$ is of type $(iii)$ otherwise we have the result by Corollary \ref{new_corollary}. We may also assume that neither $G$ nor $\overline{G}$ is critically $2$-connected, otherwise the result follows by Proposition \ref{case_1} or Proposition \ref{new_case1_james_work}. It is now clear that $V_g$ is non-empty. Therefore, by Lemma \ref{first_new_lemma_more_improved}, $|V_g|=1$ or $|V(G)| \leq 16$. If $|V(G)| \leq 16$, then we have our result. Therefore we may assume that $|V_g|=1$. It now follows that $G$ has a unique green $2$-cut $\{x,y\}$. Thus every vertex not in $\{x,y\}$ is in a red $2$-cut. As $\overline{G}$ is not critically $2$-connected, we may assume that $\overline{G} -\{x\}$ is $2$-connected. Note that $G-\{x,y\}$ has a non-trivial component $A$ and a trivial component, say $\{v\}$.

\begin{sublemma}
\label{t_vertex_sublemma1}
There is no vertex $t$ in $A$ such that $\overline{G} -\{x,v,t\}$ is connected and each of $x$ and $y$ has at least two neighbours in  $\overline{G} -\{t\}$ .
\end{sublemma}

Assume that this fails. Since $\overline{G}-\{x,v,t\}$ is connected and $v$ is adjacent to all vertices of $\overline{G}-\{x,t\}$ except $y$, we conclude that $\overline{G}-\{x,t\}$ is $2$-connected as $\overline{G} - x$ is $2$-connected and $y$ has at least two neighbours in $\overline{G} - \{x,t\}$. It now follows that $\overline{G}-\{t\}$ is $2$-connected since $x$ has at least two neighbours $\overline{G}-\{t\}$. This is a contradiction since $t$ is in a red $2$-cut.

\begin{sublemma}
\label{new_sublemma_edits7}
$\overline{G}[A]$ is connected.
\end{sublemma}

To show this, assume $\overline{G}[A]$ is disconnected. Because $\overline{G}-x$ is $2$-connected, $\overline{G}-\{x,v\}$ is connected. Since $\overline{G}[A] = \overline{G}-\{x,v,y\}$, it follows that $y$ has a neighbour in each component of $\overline{G}[A]$. As $\overline{G}-\{x,v\}$ is connected, there is a vertex $t$ in $A$ such that $\overline{G}-\{x,v,t\}$ is connected where, if possible, $t$ is chosen from a component of $\overline{G}[A]$ with at least two vertices. By \ref{t_vertex_sublemma1}, $t$ is a red neighbour of some $z$ in $\{x,y\}$ such that $z$ has degree two in $\overline{G}$. Suppose $z=y$. Then, as $\overline{G}-\{x,v,t\}$ is connected, $y$ is adjacent to $t$ and to each component of $\overline{G}[A]$, we deduce that $\{t\}$ is a component of $\overline{G}[A]$ and $|V(A)|=2$. Thus $|V(G)|=5$ and so, as $G$ is a non-$2$-cograph, $G$ is a $5$-cycle, a contradiction. We deduce that $z=x$ and $x$ has red degree two. Thus $\overline{G}-\{x,v\}$ has exactly two vertices $t$ for which $\overline{G}-\{x,v,t\}$ is connected, and each such vertex is a red neighbour of $x$. It follows that $\overline{G}-\{x,v\}$ is a path and the leaves of this path are the neighbours of $x$ in $\overline{G}-\{v\}$. Therefore $\overline{G}-\{v\}$ is a cycle, a contradiction.

Similar to \ref{t_vertex_sublemma1}, we have the following.

\begin{sublemma}
\label{t_vertex_sublemma2}
There is no vertex $t$ in $A$ such that $\overline{G}-\{y,v,t\}$ is connected and each of $x$ and $y$ has at least two neighbours in $\overline{G}-\{t\}$. 
\end{sublemma}
Assume that this fails. If $x$ has at least two neighbours in $\overline{G}-\{y,t\}$, then the proof follows as in \ref{t_vertex_sublemma1} by interchanging $x$ and $y$. Therefore we may assume that $x$ has exactly one neighbour in $\overline{G}-\{y,t\}$. Thus $\overline{G}[A] - \{t\}$ is connected and so $\overline{G}-\{x,y,t\}$ is $2$-connected. Since each of $x$ and $y$ has at least two neighbours in $\overline{G}-\{t\}$, we conclude that $\overline{G}-\{t\}$ is $2$-connected, a contradiction.

We call a vertex $t$ of $\overline{G}[A]$ \textbf{deletable} if $\overline{G}[A]-\{t\}$ is connected. By combining \ref{t_vertex_sublemma1} and \ref{t_vertex_sublemma2}, we obtain the following.

\begin{sublemma}
\label{new_sublemma_edits8}
A deletable vertex $t$ of $\overline{G}[A]$ is a neighbour in $\overline{G}$ of some $z$ in $\{x,y\}$ where $z$ has degree two in $\overline{G}$.
\end{sublemma}

\begin{sublemma}
\label{new_sublemma_edits9}
The number of deletable vertices in $\overline{G}[A]$ is in $\{2,3,4\}$.
\end{sublemma}

To see this, first observe that, since $\overline{G}[A]$ is connected having at least two vertices, it has at least two deletable vertices. Now suppose that $\overline{G}[A]$ has at least five deletable vertices. Then there is such a vertex $t$ so that, in $\overline{G}-\{t\}$, each of $x$ and $y$ has degree at least two. As $\overline{G}-\{x,v,t\}$ is connected, we have a contradiction to \ref{t_vertex_sublemma1}. Thus \ref{new_sublemma_edits9} holds.

The rest of the proof treats the three possibilities for the number of deletable vertices of $\overline{G}[A]$. First suppose that $\overline{G}[A]$ has exactly two deletable vertices $s$ and $t$. Then $\overline{G}[A]$ is a path, which we may assume has at least five vertices. Let $s'$ and $t'$ be the respective neighbours of $s$ and $t$ in $\overline{G}[A]$. Note that if either $x$ or $y$ has red neighbourhood $\{s,t\}$, then we have an induced red cycle of size at least six, which is a contradiction. Thus, by \ref{t_vertex_sublemma1} and \ref{t_vertex_sublemma2}, we may assume that both $x$ and $y$ have red degree two, and $s$ is a red neighbour of $x$, and $t$ is a red neighbour of $y$. If $xy$ is red, then $\overline{G}$ has an induced cycle of length at least seven, a contradiction. Thus both the red neighbours of $x$ and $y$ are in $A$. We show next that the respective red neighbourhoods of $x$ and $y$ are $\{s,s'\}$ and $\{t,t'\}$. To see this, let $\{s,w\}$ be the neighbourhood of $x$ in $\overline{G}$ and suppose $w \neq s'$. If $s'$ is not a red neighbour of $y$, then $\overline{G} - \{y,v,s'\}$ is connected and we get a contradiction to \ref{t_vertex_sublemma1}. Taking $z$ to be a vertex of $A$ not in $\{s,t,w,s'\}$, we see that $\overline{G}-\{x,v,z\}$ or $\overline{G}-\{y,v,z\}$ is connected and we get a contradiction to \ref{t_vertex_sublemma1} or \ref{t_vertex_sublemma2}. We conclude that $\{s,s'\}$ is the red neighbourhood of $x$. By symmetry, $\{t,t'\}$ is the red neighbourhood of $y$. Thus $\overline{G}$ is isomorphic to $L_k$ for some $k\geq 1$.

Next suppose that $\overline{G}[A]$ has exactly three deletable vertices, $s, t,$ and $u$. Then $\overline{G}[A]$ has a spanning tree $T$ having $s,t,$ and $u$ as its leaves. By \ref{new_sublemma_edits8}, each vertex in $\{s,t,u\}$ is adjacent to a red-degree-$2$ vertex in $\{x,y\}$. Moreover, neither $x$ nor $y$ has red degree exceeding two, and $xy$ is not red. Now $\overline{G}[A]$ is connected, so $\overline{G}-\{x,y\}$ is $2$-connected. As $xy$ is red, it follows that $\overline{G}-y$ is $2$-connected. Recall that we already know that $\overline{G}-x$ is $2$-connected. By symmetry, we may assume that the red neighbourhood of $x$ is $\{s,t\}$, and so $u$ is a red neighbour of $y$. Let $u'$ be the red neighbour of $u$ in $T$. Then the red neighbourhood of $y$ is $\{u,u'\}$ otherwise $\overline{G}- \{x,v,u'\}$ is connected and we get a contradiction to \ref{t_vertex_sublemma1}.  Similarly, the distance between $s$ and $t$ in $T$ is two otherwise we get a contradiction to \ref{t_vertex_sublemma2}. As $\overline{G}[A]$ has exactly three deletable vertices, the only possible edge in $\overline{G}[A]$ that is not in $T$ is $st$. Thus $\overline{G}$ is isomorphic to $M_k$ or $M_k'$ for some $k \geq 1$.

Finally, suppose that $\overline{G}[A]$ has four deletable vertices, $s, t, u,$ and $z$. We may assume that the respective red neighbourhoods of $x$ and $y$ are $\{s,t\}$ and $\{u,z\}$. Again let $T$ be a spanning tree of $\overline{G}[A]$ such that $s,t,u,$ and $z$ are leaves of $T$. Note that the distance between $s$ and $t$, and $u$ and $z$ in $T$ is two. Thus $\overline{G}$ is isomorphic to $N_k, N_k'$, or $N_k''$ for some $k\geq 1$.
\end{proof}

\section{Induced-minor-minimal non-$2$-cographs whose complements are also induced-minor-minimal non-$2$-cographs}

In this section, we consider $\mathcal{G}$, the class of induced-minor-minimal non-$2$-cographs $G$ such that $\overline{G}$ is also an induced-minor-minimal non-$2$-cograph. We show that all graphs in $\mathcal{G}$ have at most ten vertices. We give an exhaustive list of all these graphs in the appendix. We begin the section with the following
immediate consequence of Lemma~ \ref{path_three}.

\begin{corollary}
\label{induced_on_V_g_and_V_r}
Let $G$ be a graph in $\mathcal{G}$ such that $|V(G)| > 10$. Then the graph induced on the vertex set $V_g$ is a complete red graph and the graph induced on $V_r$ is a complete green graph. 
\end{corollary}

The next lemma shows that if the number of vertices of a graph $G$ in $\mathcal{G}$ exceeds ten, then $V(G) - (V_g\cup V_r)$ is non-empty.

\begin{lemma}
\label{not_V_g_union_V_r}
Let $G$ be a graph in $\mathcal{G}$ such that $|V(G)| > 10$. Then $V(G) \neq V_g\cup V_r$.
\end{lemma}

\begin{proof}

Assume that $V(G) = V_g \cup V_r$. There are $2|V_g|$ green edges and $2|V_r|$ red edges joining a vertex in $V_g$ to vertex in $V_r$. Thus 

\begin{sublemma}
\label{equation1}
$2|V_g| + 2|V_r| = |V_g||V_r|$.
\end{sublemma}

We may assume that $|V_g| \leq |V_r|$. If $|V_g| = |V_r|$, then $4|V_r| = |V_r|^2$, so $|V_r| =4$, a contradiction. Therefore $|V_g| \leq |V_r| -1$ so, by \ref{equation1}, $|V_g| |V_r| \leq 4|V_r| - 2$. Thus $|V_g| \leq 3$. If $|V_g| =3$, then, by \ref{equation1}, $|V_r| = 6$, so $|V(G)| = 9$, a contradiction. If $|V_g| \leq 2$, then we contradict \ref{equation1}. 
\end{proof}

Next we note a useful observation about the vertices in $V(G)- (V_g\cup V_r)$.

\begin{lemma}
\label{vertices_outside_V_g_union_V_r}
Let $G$ be a graph in $\mathcal{G}$ such that $|V(G)| > 10$. Then every vertex in $V(G) - (V_g\cup V_r)$ either has a green neighbour in $V_g$ or a red neighbour in $V_r$.
\end{lemma}

\begin{proof}
Since every vertex of $G$ is in either a red 2-cut or a green 2-cut, the lemma follows by Proposition \ref{bound_10_first}.
\end{proof}

\begin{lemma}
\label{V_r empty}
Let $G$ be a graph in $\mathcal{G}$ such that $|V(G)| > 10$. Then neither $V_g$ nor $V_r$ is empty.
\end{lemma}

\begin{proof}
 It suffices to show that $V_r$ is non-empty. Assume the contrary. By Lemma~ \ref{vertices_outside_V_g_union_V_r}, every vertex outside $V_g$ has a green neighbour in $V_g$. Thus, by Lemma~ \ref{vertices_outside_V_g_union_V_r_second}, $11 \leq 3|V_g|$, so $|V_g| \geq 4$. Let $\{r_1, r_2\}$ be a red 2-cut $T$. Since $V_r$ is empty, applying Proposition \ref{bound_10_first} to $\overline{G}$ gives that $T$ is contained in $V_g$. Let $v$ be a vertex in $V_g-T$ and let $\alpha$ and $\beta$ be the two green neighbours of $v$. Consider the graph $\overline{G} - T$. Note that $\overline{G} - T$ is disconnected and $v$ is incident to all the vertices in this graph except $\alpha$ and $ \beta$. Let $X$ be the component of $\overline{G} - T$ containing $v$. Since the red graph $\overline{G}$ has no degree-two vertices, $\overline{G} - T$ has exactly two components. The second component must have $\{\alpha, \beta\}$ as its vertex set. 
 
 Let $w$ be a vertex in $V_g- T - v$. As $w$ is in a different component of $\overline{G} - T$ from $\alpha$ and $\beta$, both $w \alpha$ and $w \beta$ are green edges. Since $w$ has green degree two, it follows that $\{\alpha, \beta \}$ is the green neighbourhood of each vertex in $V_g - T$. By Lemma~ \ref{vertices_outside_V_g_union_V_r}, each vertex in $V(G)-V_g-\{\alpha, \beta\}$ has a green neighbour in $V_g$. This neighbour is not in $V_g-T$, so it is in $T$. Thus $|V(G) - V_g - \{\alpha, \beta\}| \leq 4$. But $|V(G)| > 10$, so $|V_g - T| \geq 3$. Therefore $G-v$ and $\overline{G}-v$ are both $2$-connected, a contradiction. We conclude that $V_r$ is non-empty. 
\end{proof}

We are now ready to prove the second main result of the paper.

\begin{proof}[Proof of Theorem \ref{main}]

Assume that $G \in \mathcal{G}$ and $|V(G)| > 10$. Without loss of generality, let $|V_g| \leq |V_r|$. By Lemma \ref{vertices_outside_V_g_union_V_r}, every vertex in $V(G)-(V_g \cup V_r)$ either has a green neighbour in $V_g$ or a red neighbour in $V_r$. By Lemmas \ref{V_r empty} and \ref{V_g_or_V_r_has_bound_four}, $1 \leq |V_g| \leq 4$.
 Suppose $|V_g| = 4$. Then, by Lemma \ref{V_g_or_V_r_has_bound_four}, $|V_r| = 4$. Lemma~ \ref{vertices_outside_V_g_union_V_r_second} implies that $V(G) - (V_g\cup V_r)$ is empty. Therefore $|V(G)| =8$, a contradiction.

Next we assume that $|V_g| = 3$. Then every vertex in $V_r$ is a green neighbour of at least one vertex in $V_g$. Thus $|V_r| \leq 6$ as there are exactly six green edges incident to vertices in $V_g$. Then, by Lemma~ \ref{vertices_outside_V_g_union_V_r_second}, as $|V_g|=3$, we deduce that $11 \leq 3|V_g|$, a contradiction.

Now suppose that $|V_g| = 2$. Then, by Lemma~ \ref{vertices_outside_V_g_union_V_r_second}, $|V_r| \geq 5$. Let $V_g= \{u,v\}$. Since there are only four green edges meeting $V_g$, there is a vertex $w$ in $V_r$ whose red neighbours are $u$ and $v$. Thus $\{u,v\}$ is a red 2-cut. Suppose that $V_r -\{w\}$ contains at least two vertices that are joined to both $u$ and $v$ by red edges. Then one can check that both $G-w$ and $\overline{G}-w$ are $2$-connected, a contradiction. Thus $V_r$ has at most two vertices that are joined to both $u$ and $v$ by red edges. Therefore $|V_r| \leq 6$ since $V_g$ meets only four green edges. Assume that $|V_r| = 6$. Then all the green neighbours of $u$ and $v$ are in $V_r$ and are distinct. Since $|V(G)| \geq 11$, we see that $|V(G) - (V_g\cup V_r)| \geq 3$.  Let $\{w,x\}$ be the vertices in $V_r$ having both $u$ and $v$ as their red neighbours. All the vertices in $V_r - \{w,x\}$ have one red neighbour in $V_g$. Since $|V(G) - (V_g\cup V_r)| \geq 3$, Lemma~ \ref{vertices_outside_V_g_union_V_r} implies that each vertex in $V(G) - (V_g\cup V_r)$ has at most two red neighbours in $V_r - \{w,x\}$ and thus has at least two green neighbours in $V_r - \{w,x\}$. Thus $G-w$ and $\overline{G} -w$ are $2$-connected, a contradiction. We may now assume that $|V_r| = 5$ and $|V(G) - (V_g\cup V_r)| \geq 4$. By Lemma \ref{vertices_outside_V_g_union_V_r_second}, $|V(G) - (V_g \cup V_r)| = 4$. Thus, as equality holds in \ref{bound_S_new}, every vertex in $V(G) - (V_g \cup V_r)$ has at most one red neighbour in $V_r - w$ and so has at least three green neighbours in $V_r-w$. Therefore we again have that both $G-w$ and $\overline{G}-w$ are $2$-connected, a contradiction.

Finally, assume that $|V_g| = 1$. By Lemma~ \ref{vertices_outside_V_g_union_V_r_second}, $|V_r| \geq 4$. Let $V_g= \{v\}$ and let $\alpha \in V_r$ be a red neighbour of $v$. First, we show that $V_r$ does not contain a green 2-cut that contains $\alpha$. Assume that $\{\alpha, \beta\}$ is a green 2-cut where $\{\alpha, \beta\} \subseteq V_r$. Then $G - \{\alpha, \beta\}$ has a component $X$ that contains $V_r - \{\alpha, \beta\}$ and all but at most two vertices of $V(G)-\{\alpha, \beta\}$. Let $Y$ be a component of $G - \{\alpha, \beta\}$ different from $X$. Then $|V(Y)| \leq 2$. Suppose $|V(Y)|=1$. Then the vertex in $Y$ must be in $V_g$, so it is $v$. This is a contradiction since $\alpha v$ is red. Thus $|V(Y)|=2$ and $G - \{\alpha, \beta\}$ has exactly two components. Then $|V(X)| \geq 7$. Let $x$ be a vertex in $X$ such that $x$ is not a red neighbour of $\alpha$ or $\beta$, and $X - \{x\}$ contains at least two vertices of $V_r - \{\alpha, \beta\}$. Since each vertex of $V_r - \{\alpha, \beta\}$ has its two red neighbours in $Y$ and so is adjacent in $G$ to every vertex of $X$, it follows that $G-x$ is $2$-connected. Moreover, by Lemma \ref{complement_connected_observation}, $\overline{G}-x$ is $2$-connected, a contradiction. We conclude that $V_r$ does not have a green $2$-cut containing $\alpha$.

Next, we show that no green $2$-cut contains $\alpha$. Assume that $\{\alpha, z\}$ is a green 2-cut. Then $z \notin V_r$. By Proposition $\ref{bound_10_first}$, $G - \{\alpha, z\}$ has a single-vertex component $Y$. Since the vertex in $Y$ has green degree two, $Y = \{v\}$. Thus $\alpha v$ is green, a contradiction. We conclude that deleting from $G$ any red neighbour of $v$ in $V_r$ leaves a green graph that is still $2$-connected.

To complete the proof of the theorem, we show that $v$ has a red neighbour in $V_r$ whose deletion from $\overline{G}$ leaves a $2$-connected graph, thus arriving at a contradiction. Let $\beta$ be a red neighbour of $v$ in $V_r - \{\alpha\}$. If $\alpha$ and $\beta$ have the same red neighbourhood, say $\{x,v\}$, then $\{x,v\}$ is a red 2-cut and we obtain a contradiction by applying Proposition $\ref{bound_10_first}$ to $\overline{G}$. Thus $\alpha$ and $\beta$ have distinct red neighbourhoods, $\{x,v\}$ and $\{y,v\}$, respectively. Note that if $xv$ is red, then $\overline{G} - \alpha$ is $2$-connected. Thus we may assume that both $xv$ and $yv$ are green. This implies $\gamma v$ is red for each $\gamma$ in $V_r - \{\alpha, \beta\}$ since $v$ has green degree two. Thus, for some fixed $\gamma$ in $V_r - \{\alpha, \beta\}$, the other red neighbour, $z$, of $\gamma$ is distinct from $x$ and $y$. Since $vz$ is red and $\gamma$ has red degree two, we see that $\overline{G}-\gamma$ is $2$-connected, a contradiction.
\end{proof}

\section{Appendix}

We implemented the algorithm given in this section using SageMath~\cite{sage} and provide a list of all graphs in $\mathcal{G}$ up to complementation. The graphs in this section are drawn using SageMath.

\hspace{5mm}

{\bf Graphs on six vertices.} There are two graphs on six vertices in $\mathcal{G}$, the graph in Figure \ref{six} and its complement.

\begin{figure}[htbp]
\includegraphics[scale = 0.3]{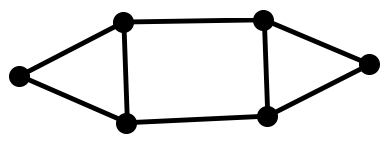}
\caption{A $6$-vertex graph in $\mathcal{G}$.}
\label{six}
\end{figure}

{\bf Graphs on seven vertices.} There are sixteen graphs on seven vertices in $\mathcal{G}$, the graphs in Figure \ref{seven} and their complements.

\begin{figure}[htbp]
\centering
\begin{minipage}{.20\linewidth}
  \includegraphics[scale=0.22]{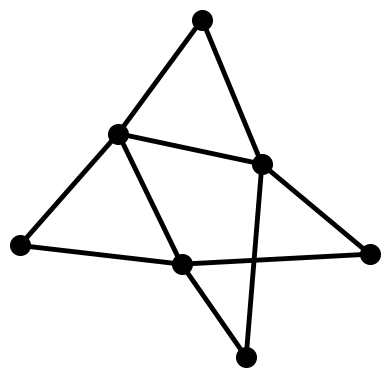}
\end{minipage}
\hspace{.01\linewidth}
\begin{minipage}{.20\linewidth}
  \includegraphics[scale=0.22]{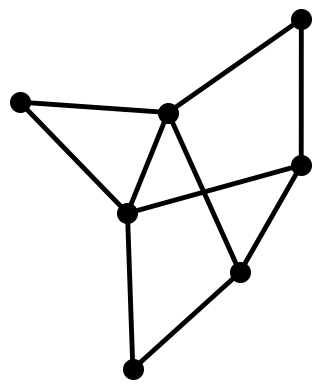}
\end{minipage} \hspace{.01\linewidth}
\begin{minipage}{.20\linewidth}
  \includegraphics[scale=0.22]{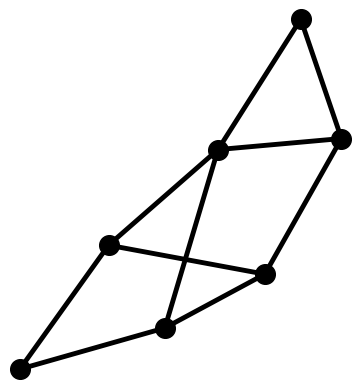}
\end{minipage}
\hspace{.01\linewidth}
\begin{minipage}{.20\linewidth}
\includegraphics[scale=0.22]{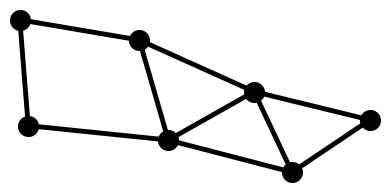}
\end{minipage}

\begin{minipage}{.20\linewidth}
  \includegraphics[scale=0.22]{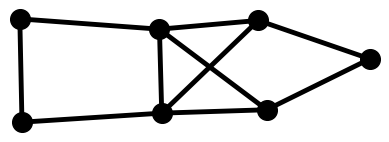}
\end{minipage}
\hspace{.01\linewidth}
\begin{minipage}{.20\linewidth}
  \includegraphics[scale=0.22]{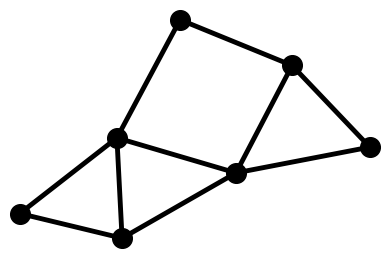}
\end{minipage} \hspace{.01\linewidth}
\begin{minipage}{.20\linewidth}
  \includegraphics[scale=0.22]{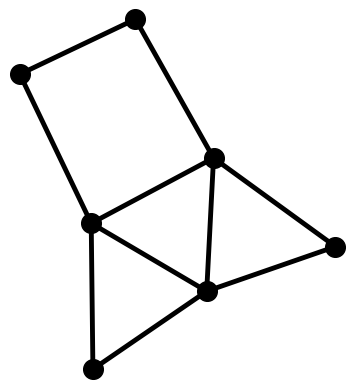}
\end{minipage}
\hspace{.01\linewidth}
\begin{minipage}{.20\linewidth}
\includegraphics[scale=0.22]{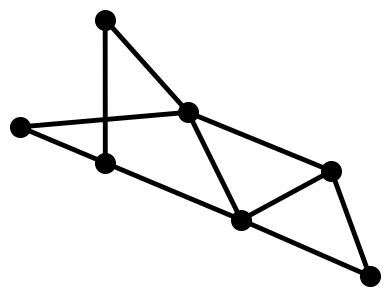}
\end{minipage}

\caption{Graphs on seven vertices in $\mathcal{G}$.}
\label{seven}

\end{figure}

{\bf Graphs on eight vertices.} There are 87 graphs on eight vertices in $\mathcal{G}$, of which five are self-complementary. Figure \ref{eight_vertices_self} shows these self-complementary graphs. Figure \ref{eight_vertices} shows  41 non-self-complementary graphs that, with their complements, are the remaining $8$-vertex graphs in $\mathcal{G}$.

\begin{figure}[htbp]
\centering
\begin{minipage}{.25\linewidth}
  \includegraphics[scale=0.25]{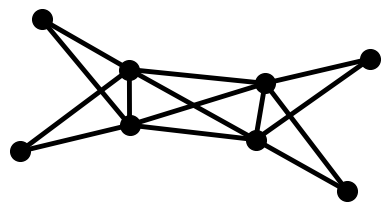}
\end{minipage}
\hspace{.02\linewidth}
\begin{minipage}{.25\linewidth}
  \includegraphics[scale=0.25]{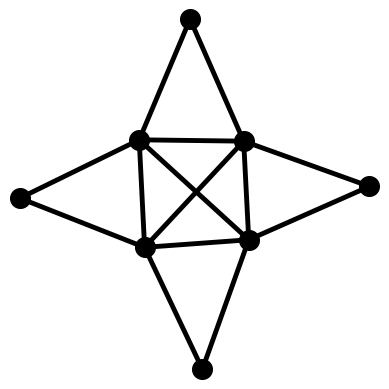}
\end{minipage} \hspace{.02\linewidth}
\begin{minipage}{.25\linewidth}
  \includegraphics[scale=0.25]{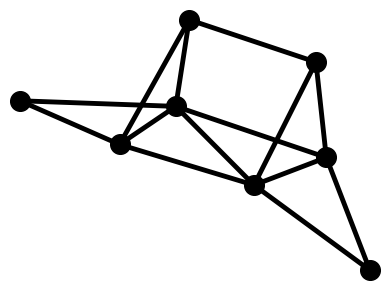}
\end{minipage}
\hspace{.02\linewidth}
\begin{minipage}{.25\linewidth}
\includegraphics[scale=0.25]{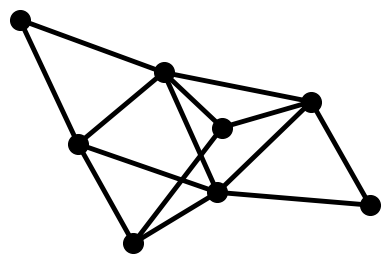}
\end{minipage} 
\hspace{.02\linewidth}
\begin{minipage}{.25\linewidth}
  \includegraphics[scale=0.25]{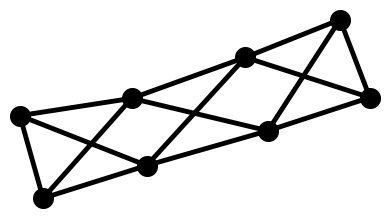}
\end{minipage}

\caption{Self-complementary graphs on eight vertices in $\mathcal{G}$.}
\label{eight_vertices_self}

\end{figure}

\begin{figure}[htbp]
\centering

\begin{minipage}{.17\linewidth}
  \includegraphics[scale=0.22]{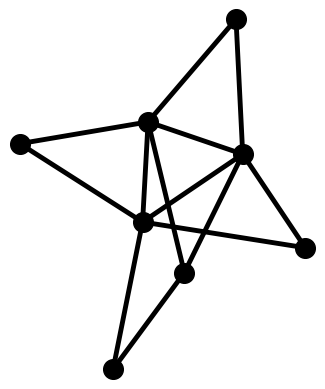}
\end{minipage} \hspace{.01\linewidth}
\begin{minipage}{.17\linewidth}
  \includegraphics[scale=0.22]{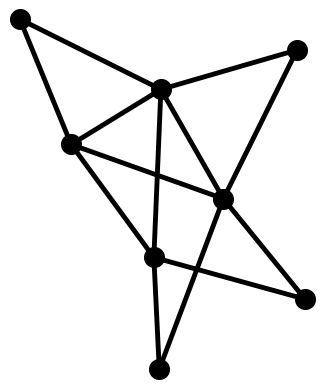}
\end{minipage}
\hspace{.01\linewidth}
\begin{minipage}{.17\linewidth}
\includegraphics[scale=0.22]{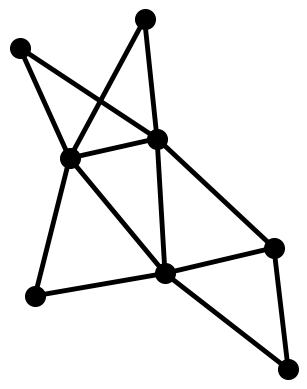}
\end{minipage}
\hspace{.01\linewidth}
\begin{minipage}{.17\linewidth}
\includegraphics[scale=0.22]{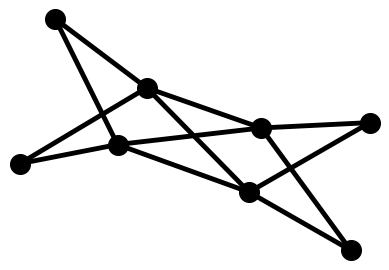}
\end{minipage}
\begin{minipage}{.17\linewidth}
  \includegraphics[scale=0.22]{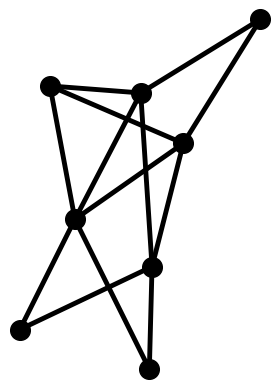}
\end{minipage}
\hspace{.01\linewidth}

\begin{minipage}{.17\linewidth}
  \includegraphics[scale=0.22]{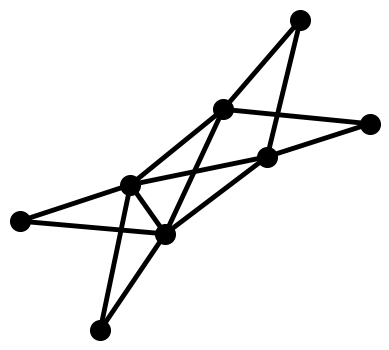}
\end{minipage} \hspace{.01\linewidth}
\begin{minipage}{.17\linewidth}
  \includegraphics[scale=0.22]{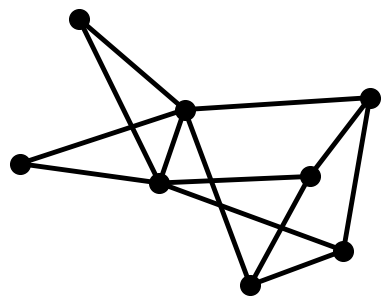}
\end{minipage}
\hspace{.01\linewidth}
\begin{minipage}{.17\linewidth}
\includegraphics[scale=0.22]{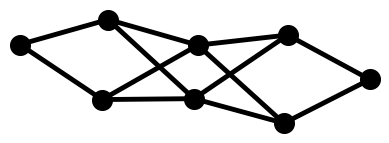}
\end{minipage}
\hspace{.01\linewidth}
\begin{minipage}{.17\linewidth}
\includegraphics[scale=0.22]{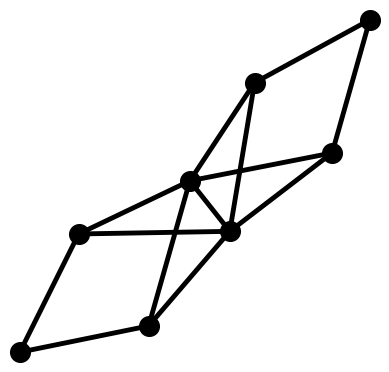}
\end{minipage}
\begin{minipage}{.17\linewidth}
  \includegraphics[scale=0.22]{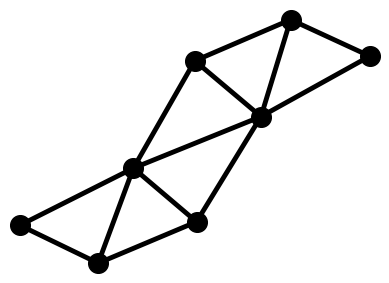}
\end{minipage}
\hspace{.01\linewidth}

\begin{minipage}{.17\linewidth}
  \includegraphics[scale=0.22]{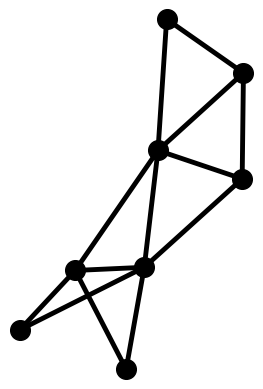}
\end{minipage} \hspace{.01\linewidth}
\begin{minipage}{.17\linewidth}
  \includegraphics[scale=0.22]{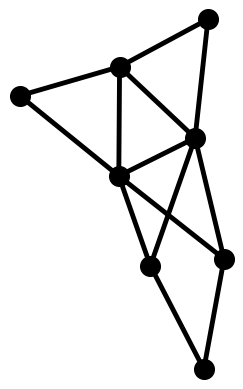}
\end{minipage}
\hspace{.01\linewidth}
\begin{minipage}{.17\linewidth}
\includegraphics[scale=0.22]{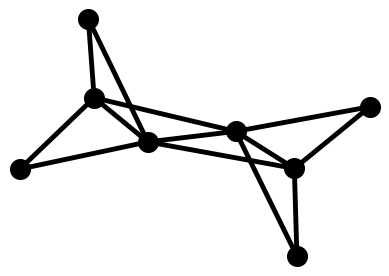}
\end{minipage}
\hspace{.01\linewidth}
\begin{minipage}{.17\linewidth}
\includegraphics[scale=0.22]{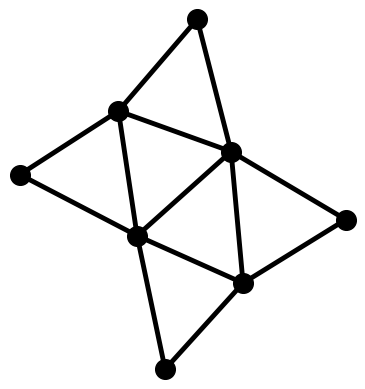}
\end{minipage}
\begin{minipage}{.17\linewidth}
  \includegraphics[scale=0.22]{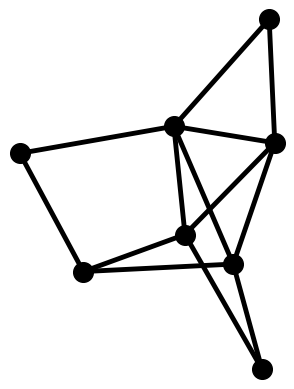}
\end{minipage}
\hspace{.01\linewidth}

\begin{minipage}{.17\linewidth}
  \includegraphics[scale=0.22]{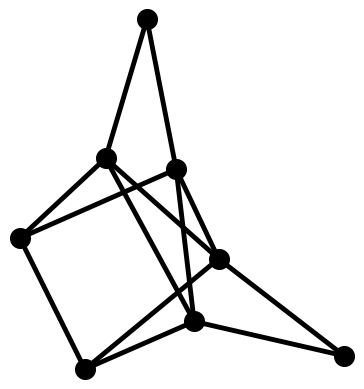}
\end{minipage} \hspace{.01\linewidth}
\begin{minipage}{.17\linewidth}
  \includegraphics[scale=0.22]{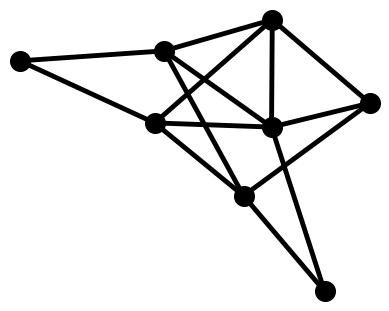}
\end{minipage}
\hspace{.01\linewidth}
\begin{minipage}{.17\linewidth}
\includegraphics[scale=0.22]{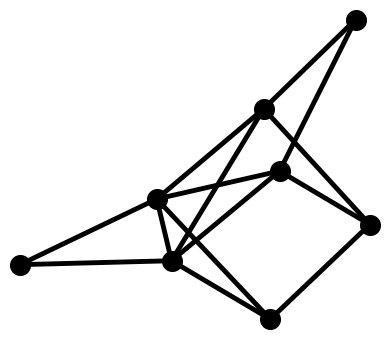}
\end{minipage}
\hspace{.01\linewidth}
\begin{minipage}{.17\linewidth}
\includegraphics[scale=0.22]{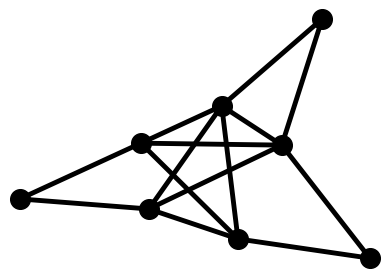}
\end{minipage}
\begin{minipage}{.17\linewidth}
  \includegraphics[scale=0.22]{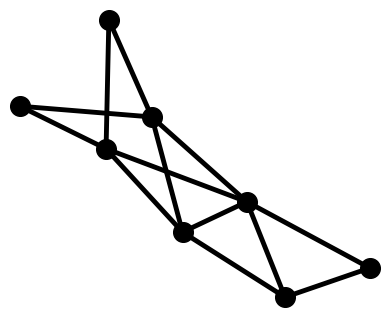}
\end{minipage}
\hspace{.01\linewidth}

\begin{minipage}{.17\linewidth}
  \includegraphics[scale=0.22]{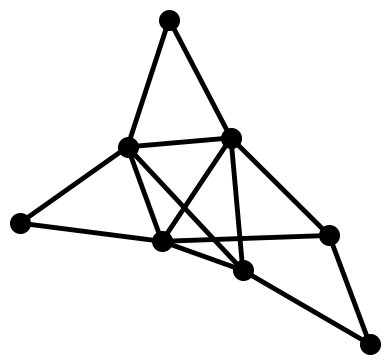}
\end{minipage} \hspace{.01\linewidth}
\begin{minipage}{.17\linewidth}
  \includegraphics[scale=0.22]{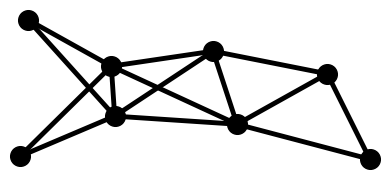}
\end{minipage}
\hspace{.01\linewidth}
\begin{minipage}{.17\linewidth}
\includegraphics[scale=0.22]{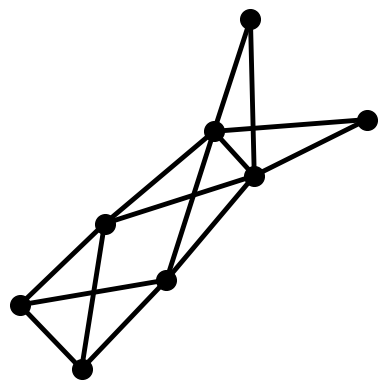}
\end{minipage}
\hspace{.01\linewidth}
\begin{minipage}{.17\linewidth}
\includegraphics[scale=0.22]{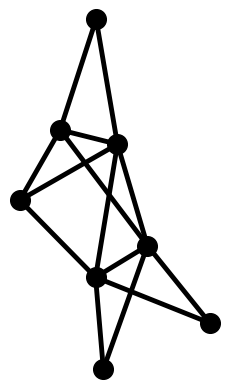}
\end{minipage}
\begin{minipage}{.17\linewidth}
  \includegraphics[scale=0.22]{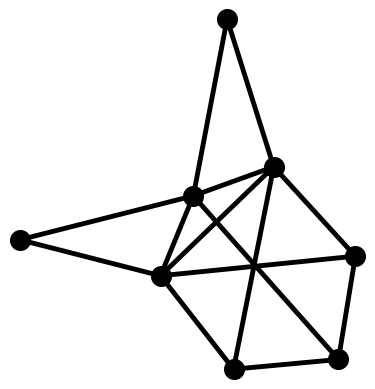}
\end{minipage}
\hspace{.01\linewidth}

\begin{minipage}{.17\linewidth}
  \includegraphics[scale=0.22]{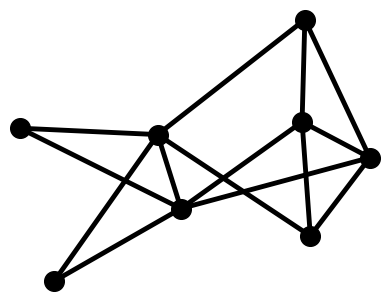}
\end{minipage} \hspace{.01\linewidth}
\begin{minipage}{.17\linewidth}
  \includegraphics[scale=0.22]{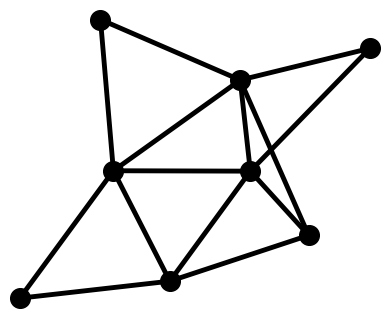}
\end{minipage}
\hspace{.01\linewidth}
\begin{minipage}{.17\linewidth}
\includegraphics[scale=0.22]{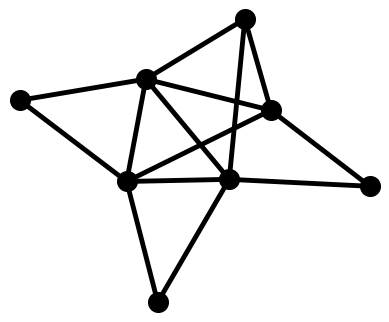}
\end{minipage}
\hspace{.01\linewidth}
\begin{minipage}{.17\linewidth}
\includegraphics[scale=0.22]{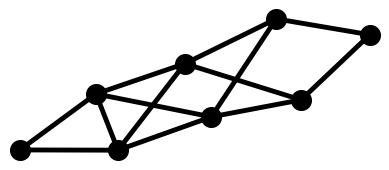}
\end{minipage}
\begin{minipage}{.17\linewidth}
  \includegraphics[scale=0.22]{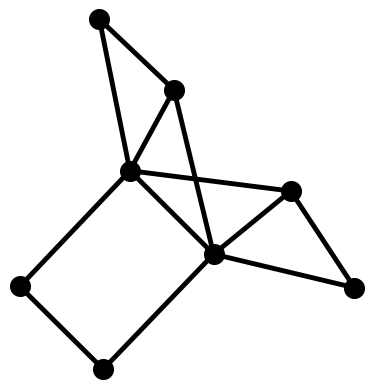}
\end{minipage}
\hspace{.01\linewidth}

\begin{minipage}{.17\linewidth}
  \includegraphics[scale=0.22]{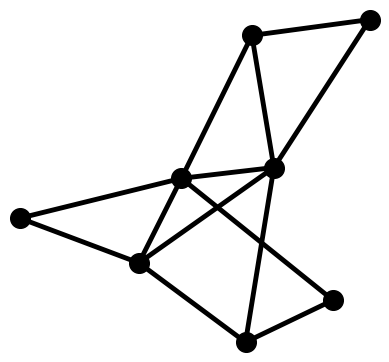}
\end{minipage} \hspace{.01\linewidth}
\begin{minipage}{.17\linewidth}
  \includegraphics[scale=0.22]{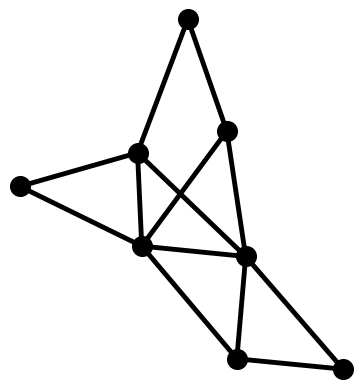}
\end{minipage}
\hspace{.01\linewidth}
\begin{minipage}{.17\linewidth}
\includegraphics[scale=0.22]{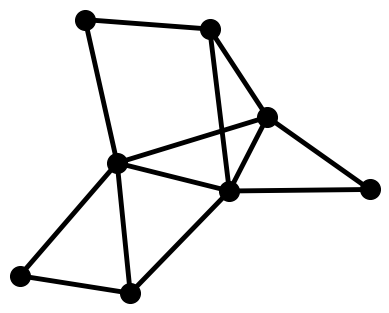}
\end{minipage}
\hspace{.01\linewidth}
\begin{minipage}{.17\linewidth}
\includegraphics[scale=0.22]{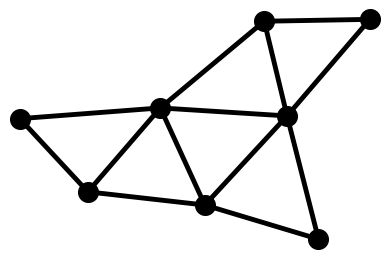}
\end{minipage}
\begin{minipage}{.17\linewidth}
  \includegraphics[scale=0.22]{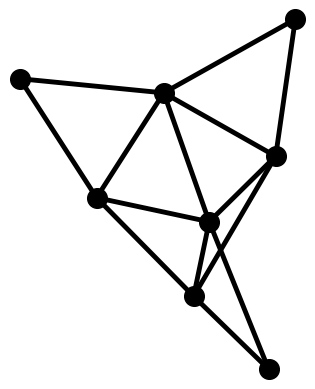}
\end{minipage}
\hspace{.01\linewidth}

\begin{minipage}{.17\linewidth}
  \includegraphics[scale=0.22]{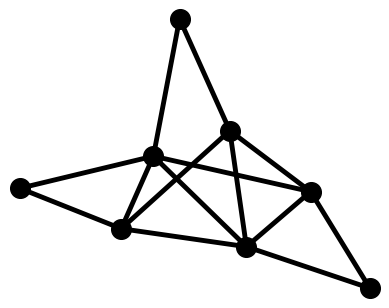}
\end{minipage} \hspace{.01\linewidth}
\begin{minipage}{.17\linewidth}
  \includegraphics[scale=0.22]{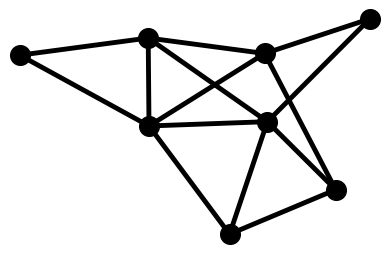}
\end{minipage}
\hspace{.01\linewidth}
\begin{minipage}{.17\linewidth}
\includegraphics[scale=0.22]{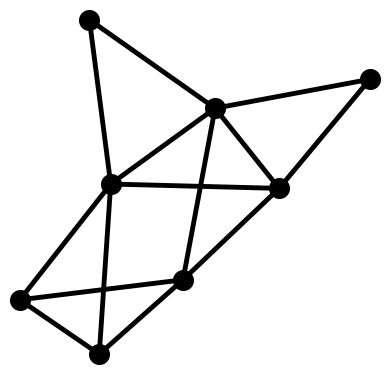}
\end{minipage}
\hspace{.01\linewidth}
\begin{minipage}{.17\linewidth}
\includegraphics[scale=0.22]{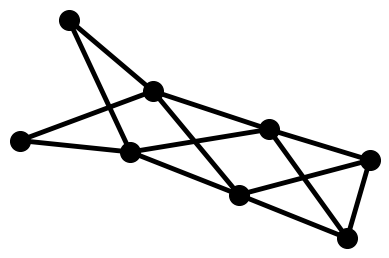}
\end{minipage}
\begin{minipage}{.17\linewidth}
  \includegraphics[scale=0.22]{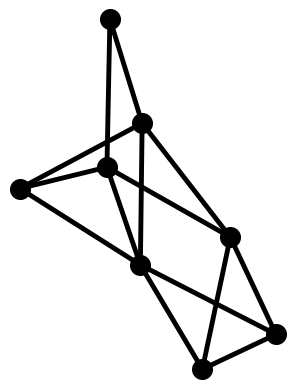}
\end{minipage}
\hspace{.01\linewidth}

\begin{minipage}{.17\linewidth}
  \includegraphics[scale=0.22]{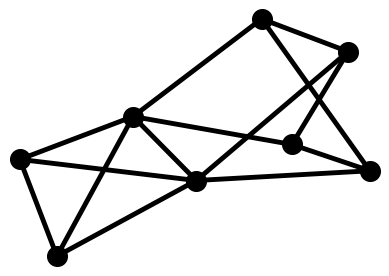}
\end{minipage}

\caption{Graphs on eight vertices in $\mathcal{G}$.}

\label{eight_vertices}

\end{figure}

{\bf Graphs on nine vertices.} There are 86 graphs on nine vertices in $\mathcal{G}$. These are the 43 graphs in Figure \ref{nine_vertices} and their complements.

\begin{figure}[htbp]
\centering

\begin{minipage}{.17\linewidth}
  \includegraphics[scale=0.22]{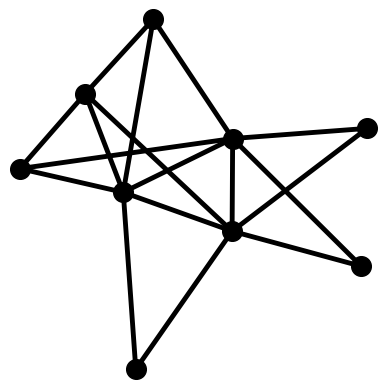}
\end{minipage} \hspace{.01\linewidth}
\begin{minipage}{.17\linewidth}
  \includegraphics[scale=0.22]{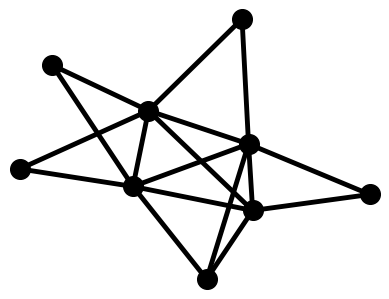}
\end{minipage}
\hspace{.01\linewidth}
\begin{minipage}{.17\linewidth}
\includegraphics[scale=0.22]{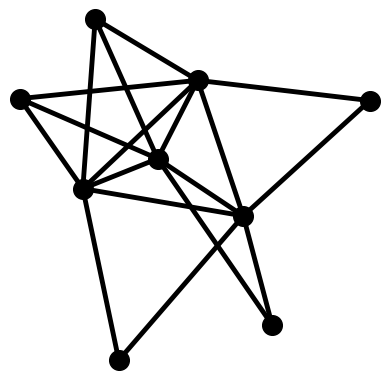}
\end{minipage}
\hspace{.01\linewidth}
\begin{minipage}{.17\linewidth}
\includegraphics[scale=0.22]{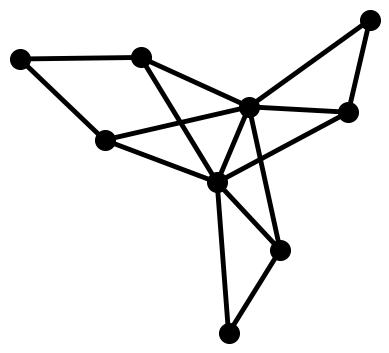}
\end{minipage}
\begin{minipage}{.17\linewidth}
  \includegraphics[scale=0.22]{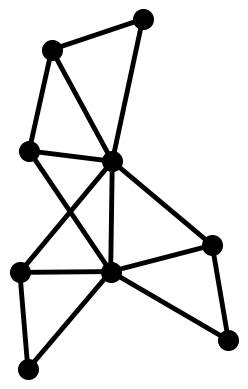}
\end{minipage}
\hspace{.01\linewidth}

\begin{minipage}{.17\linewidth}
  \includegraphics[scale=0.22]{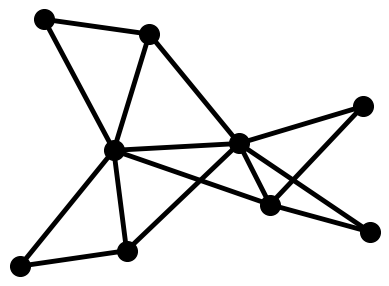}
\end{minipage} \hspace{.01\linewidth}
\begin{minipage}{.17\linewidth}
  \includegraphics[scale=0.22]{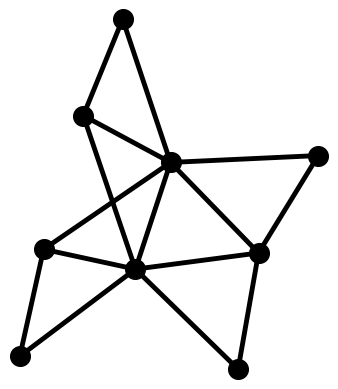}
\end{minipage}
\hspace{.01\linewidth}
\begin{minipage}{.17\linewidth}
\includegraphics[scale=0.22]{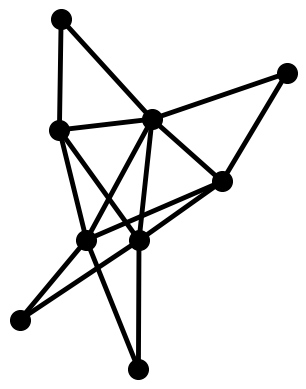}
\end{minipage}
\hspace{.01\linewidth}
\begin{minipage}{.17\linewidth}
\includegraphics[scale=0.22]{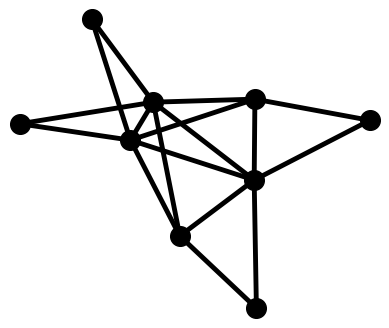}
\end{minipage}
\begin{minipage}{.17\linewidth}
  \includegraphics[scale=0.22]{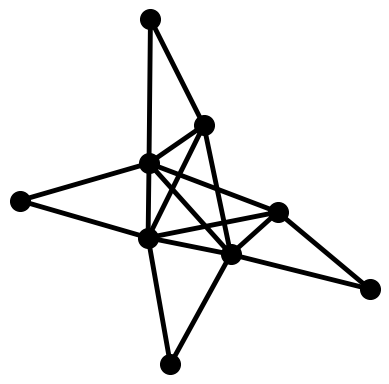}
\end{minipage}
\hspace{.01\linewidth}

\begin{minipage}{.17\linewidth}
  \includegraphics[scale=0.22]{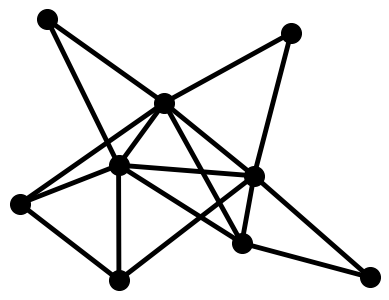}
\end{minipage} \hspace{.01\linewidth}
\begin{minipage}{.17\linewidth}
  \includegraphics[scale=0.22]{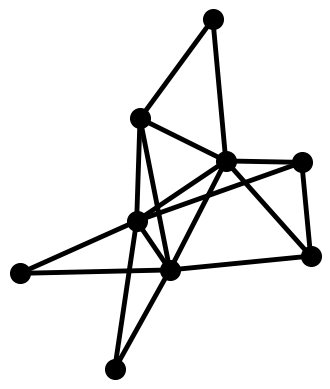}
\end{minipage}
\hspace{.01\linewidth}
\begin{minipage}{.17\linewidth}
\includegraphics[scale=0.22]{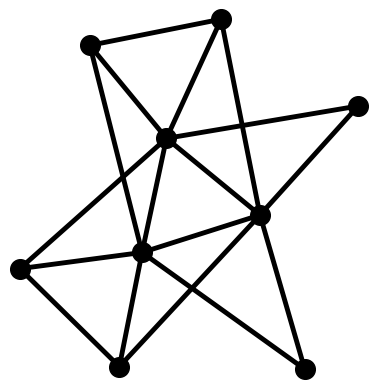}
\end{minipage}
\hspace{.01\linewidth}
\begin{minipage}{.17\linewidth}
\includegraphics[scale=0.22]{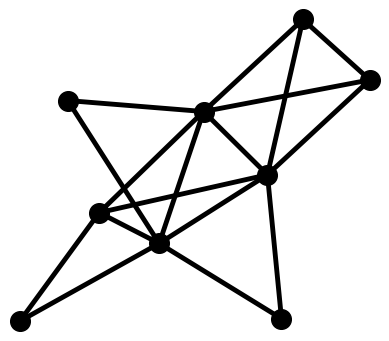}
\end{minipage}
\begin{minipage}{.17\linewidth}
  \includegraphics[scale=0.22]{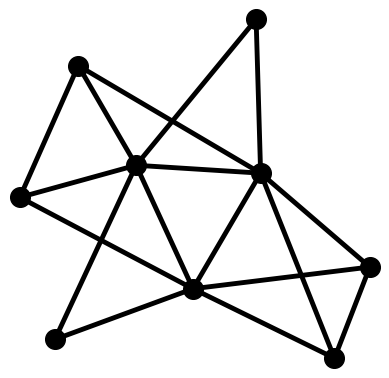}
\end{minipage}
\hspace{.01\linewidth}

\begin{minipage}{.17\linewidth}
  \includegraphics[scale=0.22]{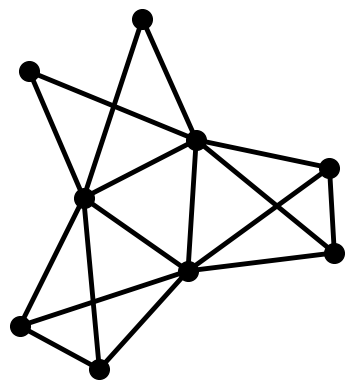}
\end{minipage} \hspace{.01\linewidth}
\begin{minipage}{.17\linewidth}
  \includegraphics[scale=0.22]{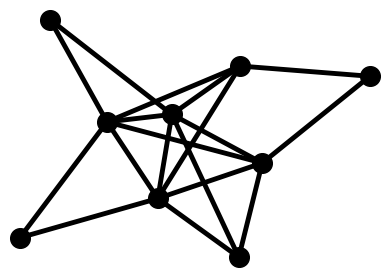}
\end{minipage}
\hspace{.01\linewidth}
\begin{minipage}{.17\linewidth}
\includegraphics[scale=0.22]{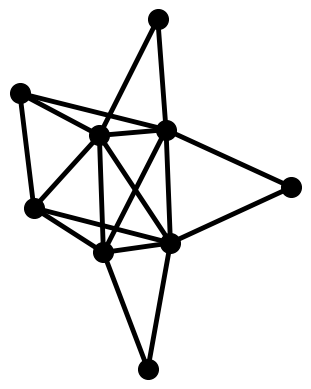}
\end{minipage}
\hspace{.01\linewidth}
\begin{minipage}{.17\linewidth}
\includegraphics[scale=0.22]{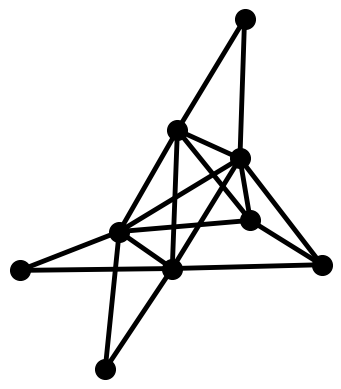}
\end{minipage}
\begin{minipage}{.17\linewidth}
  \includegraphics[scale=0.22]{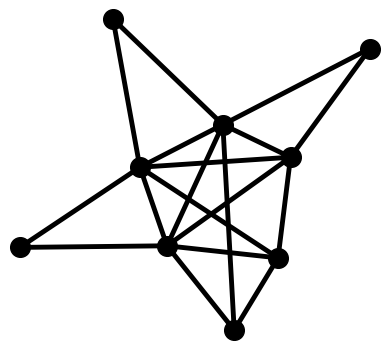}
\end{minipage}
\hspace{.01\linewidth}

\begin{minipage}{.17\linewidth}
  \includegraphics[scale=0.22]{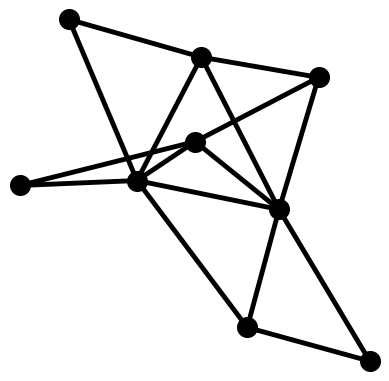}
\end{minipage} \hspace{.01\linewidth}
\begin{minipage}{.17\linewidth}
  \includegraphics[scale=0.22]{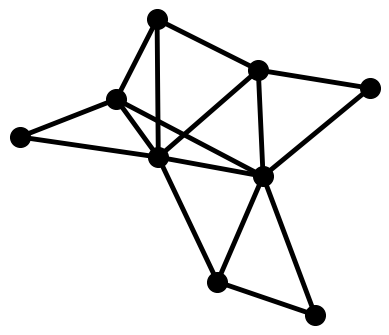}
\end{minipage}
\hspace{.01\linewidth}
\begin{minipage}{.17\linewidth}
\includegraphics[scale=0.22]{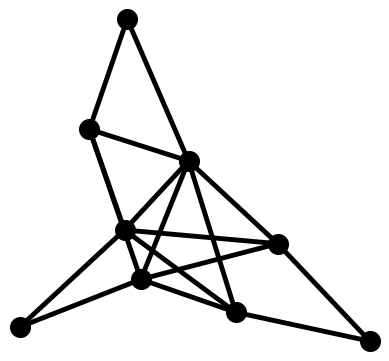}
\end{minipage}
\hspace{.01\linewidth}
\begin{minipage}{.17\linewidth}
\includegraphics[scale=0.22]{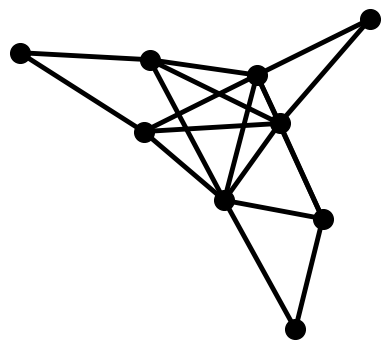}
\end{minipage}
\begin{minipage}{.17\linewidth}
  \includegraphics[scale=0.22]{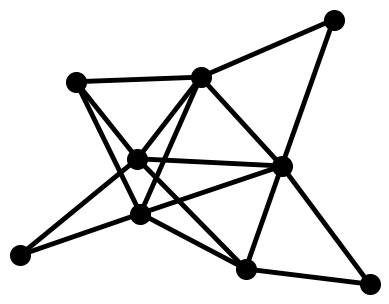}
\end{minipage}
\hspace{.01\linewidth}

\begin{minipage}{.17\linewidth}
  \includegraphics[scale=0.22]{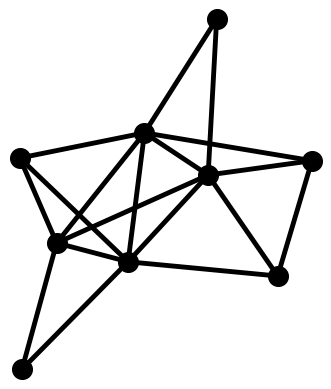}
\end{minipage} \hspace{.01\linewidth}
\begin{minipage}{.17\linewidth}
  \includegraphics[scale=0.22]{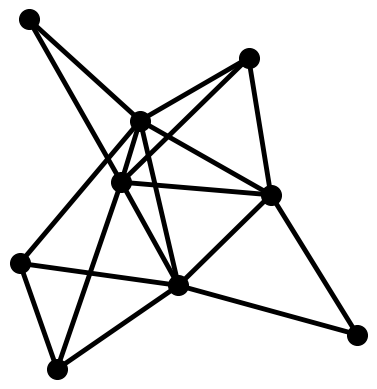}
\end{minipage}
\hspace{.01\linewidth}
\begin{minipage}{.17\linewidth}
\includegraphics[scale=0.22]{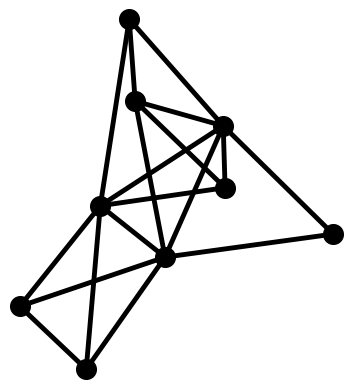}
\end{minipage}
\hspace{.01\linewidth}
\begin{minipage}{.17\linewidth}
\includegraphics[scale=0.22]{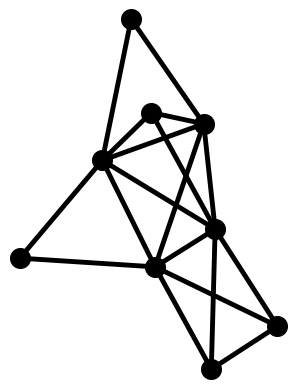}
\end{minipage}
\begin{minipage}{.17\linewidth}
  \includegraphics[scale=0.22]{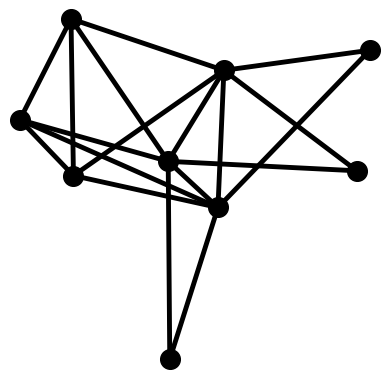}
\end{minipage}
\hspace{.01\linewidth}

\begin{minipage}{.17\linewidth}
  \includegraphics[scale=0.22]{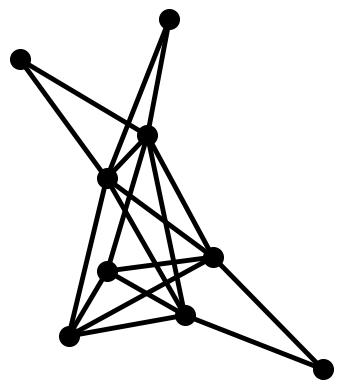}
\end{minipage} \hspace{.01\linewidth}
\begin{minipage}{.17\linewidth}
  \includegraphics[scale=0.22]{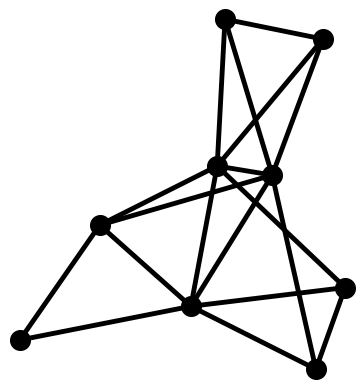}
\end{minipage}
\hspace{.01\linewidth}
\begin{minipage}{.17\linewidth}
\includegraphics[scale=0.22]{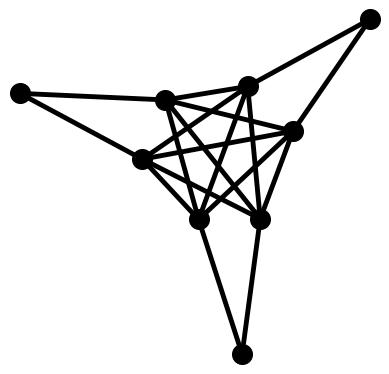}
\end{minipage}
\hspace{.01\linewidth}
\begin{minipage}{.17\linewidth}
\includegraphics[scale=0.22]{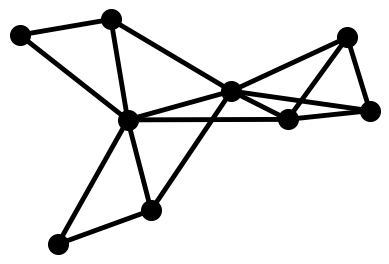}
\end{minipage}
\begin{minipage}{.17\linewidth}
  \includegraphics[scale=0.22]{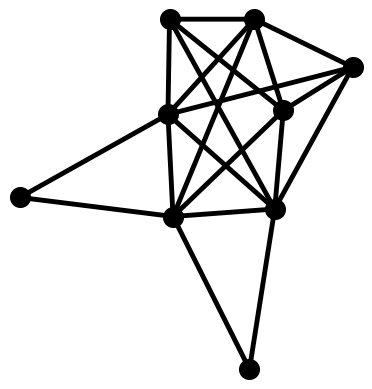}
\end{minipage}
\hspace{.01\linewidth}

\begin{minipage}{.17\linewidth}
  \includegraphics[scale=0.22]{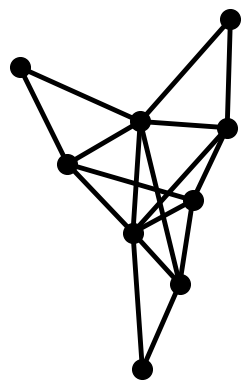}
\end{minipage} \hspace{.01\linewidth}
\begin{minipage}{.17\linewidth}
  \includegraphics[scale=0.22]{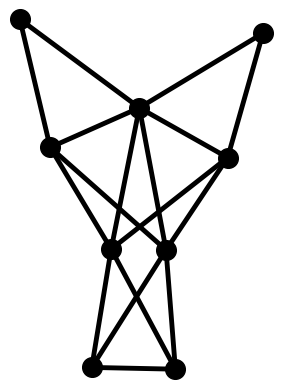}
\end{minipage}
\hspace{.01\linewidth}
\begin{minipage}{.17\linewidth}
\includegraphics[scale=0.22]{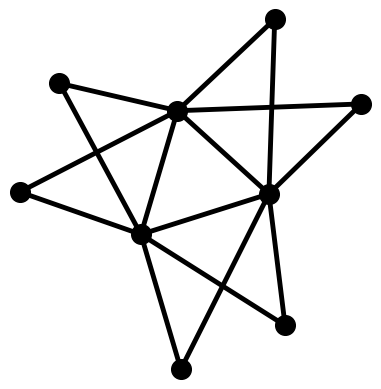}
\end{minipage}
\hspace{.01\linewidth}
\begin{minipage}{.17\linewidth}
\includegraphics[scale=0.22]{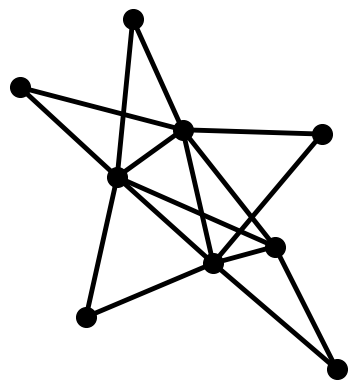}
\end{minipage}
\begin{minipage}{.17\linewidth}
  \includegraphics[scale=0.22]{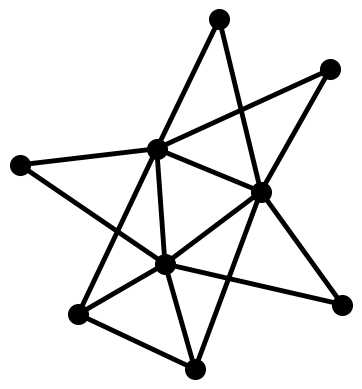}
\end{minipage}
\hspace{.01\linewidth}

\begin{minipage}{.17\linewidth}
  \includegraphics[scale=0.22]{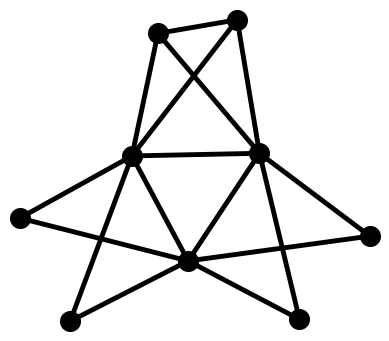}
\end{minipage} \hspace{.01\linewidth}
\begin{minipage}{.17\linewidth}
  \includegraphics[scale=0.22]{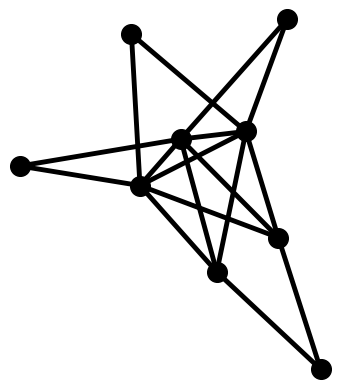}
\end{minipage} \hspace{.01\linewidth}
\begin{minipage}{.17\linewidth}
  \includegraphics[scale=0.22]{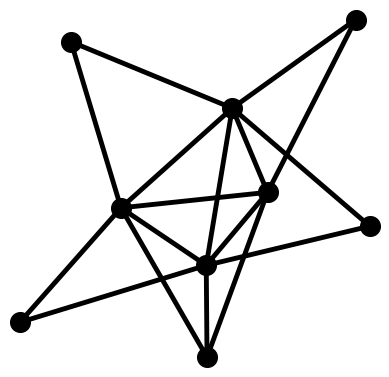}
\end{minipage}

\caption{Graphs on nine vertices in $\mathcal{G}$.}
\label{nine_vertices}
\end{figure}

{\bf Graphs on ten vertices.} There are two graphs on ten vertices in $\mathcal{G}$, the graph in Figure \ref{ten_vertices} and its complement.

\begin{figure}[htbp]
\includegraphics[scale = 0.3]{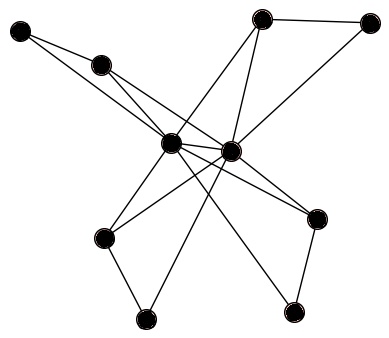}
\caption{A $10$-vertex graph in $\mathcal{G}$.}
\label{ten_vertices}
\end{figure}

\begin{algorithm}
\renewcommand{\thealgorithm}{}
\label{pseudocode2}
\caption{Finding graphs in $\mathcal{G}$ of order at most ten}

\begin{algorithmic}
\REQUIRE $n = 6,7,8,9$ or $10$.
\STATE Set FinalList $\leftarrow \emptyset$, $i \leftarrow 0$, $j \leftarrow 0$
\STATE Generate all two connected graphs of order $n$ using nauty geng \cite{nauty} and store in an iterator $L$

\FOR{$g$ in $L$ such that vertex connectivity of $g$ and $\overline{g}$ is 2}

    \FOR{$v$ in $V(g)$}
        \STATE $h=g \ba v$
        \IF{$h$ is a $2$-cograph}
            \STATE $i \leftarrow i+1$
        \ENDIF
    \ENDFOR

    \FOR{$e$ in $E(g)$}
        \STATE $h=g/e$
        \IF{$h$ is a $2$-cograph}
            \STATE $j \leftarrow j+1$
        \ENDIF
    \ENDFOR

\IF{$i$ equals $|V(g)|$ and $j$ equals $|E(g)|$}

\STATE Add $g$ to FinalList

\ENDIF

\ENDFOR

\FOR {$g$ in FinalList}
    \IF{FinalList does not contain $\overline{g}$}
     \STATE remove $g$ from FinalList
    
    \ENDIF
\ENDFOR

\end{algorithmic}
\end{algorithm}

\FloatBarrier

\section*{Acknowledgement}

The authors thank Guoli Ding for suggesting the problem and for helpful discussions on it. The authors also thank Zachary Gershkoff for help with SageMath.

\end{document}